\documentclass[11pt]{article}
\usepackage[utf8]{inputenc}  
\usepackage[T1]{fontenc} 
\usepackage{lmodern}
\usepackage{amsmath} 
\usepackage{amssymb}
\usepackage{amsthm} 
\usepackage{mathrsfs}
\usepackage{graphicx}
\usepackage{ccaption}
\usepackage{hyperref}
\usepackage{listings}
\usepackage{enumitem}
\usepackage{tikz}
\usepackage{float}
\usepackage{pgfplots}
\pgfplotsset{compat=newest}
\usetikzlibrary{shapes}
\usetikzlibrary{patterns}
\usepackage{stmaryrd}
\usepackage{authblk}
\usepackage[pagewise]{lineno}
\usepackage{color}

\usepackage[left=2.5cm,right=2.5cm,top=3cm,bottom=3cm]{geometry}

\theoremstyle{definition}
\newtheorem{defi}{Definition}
\newtheorem{theorem}{Theorem}
\newtheorem{prop}{Proposition}
\newtheorem{asum}{Assumption}
\newtheorem{coro}{Corollary}
\newtheorem{lem}{Lemma}
\newtheorem{rem}{Remark}
\newcommand{\R}{\mathbb{R}}
\newcommand{\N}{\mathbb{N}}

\newcommand{\C}{\mathbb{C}}

\newcommand{\dd}{\mathrm{d}}

\newcommand{\eps}{\varepsilon}
\newcommand{\loc}{\text{loc}}

\newcommand{\Ai}{\mathcal{A}}
\newcommand{\Bi}{\mathcal{B}}

    \pgfplotsset{
        colormap={parula}{
            rgb255=(53,42,135)
            rgb255=(15,92,221)
            rgb255=(18,125,216)
            rgb255=(7,156,207)
            rgb255=(21,177,180)
            rgb255=(89,189,140)
            rgb255=(165,190,107)
            rgb255=(225,185,82)
            rgb255=(252,206,46)
            rgb255=(249,251,14)
        },
    }
\everymath{\displaystyle}   

\newcommand\tint{\textstyle\int\nolimits}

\newcommand{\g}[1]{\boldsymbol{#1}}

\def\ph{\varphi}
\def\top{\text{top}}
\def\bot{\text{bot}}
\def\app{\text{app}}

\def\cH{\mathcal{H}}

\def\cO{\mathcal{O}}
\def\tH{\text{H}}
\def\tL{\text{L}}
\DeclareMathOperator\supp{\text{supp}}

\title{Layer stripping approach to reconstruct shape defects in waveguides using locally resonant frequencies}
\author[1,*]{Angèle Niclas}
\author[2]{Laurent Seppecher}
\affil[1]{Centre de Mathématiques Appliquées, \'Ecole Polytechnique, France}
\affil[2]{Institut Camille Jordan, \'Ecole Centrale Lyon, France}
\affil[*]{Corresponding author: angele.niclas@polytechnique.edu}
\date{}

\begin{document}
\maketitle 

\begin{abstract}
This article present a new method to reconstruct slowly varying width defects in 2D waveguides using one-side section measurements at locally resonant frequencies. At these frequencies, locally resonant modes propagate in the waveguide up to a ``cut-off'' position. In this particular point, the local width of the waveguide can be recovered. Given multi-frequency measurements taken on a section of the waveguide, we perform an efficient layer stripping approach to recover shape variations slice by slice. It provides an $\text{L}^\infty$-stable method to reconstruct the width of a slowly monotonous varying waveguide. We validate this method on numerical data and discuss its limits.  
\end{abstract}


\section{Introduction}

In previous work \cite{niclas1}, we have presented the analysis of recovering a slowly varying shape of waveguides from surface measurement using locally resonant frequencies. This method was suited for waveguides such as elastic plates where surface measurements are available but is not appropriate for acoustic waveguides such as pipes or air ducts where data are often measured on one section of the waveguide \cite{kharrat1,bourgeois1}.

In the present work, we tackle the shape recovery problem from one-side section measurements only. At locally resonant frequencies, the usual Born approximation (see \cite{colton1,bonnetier1}) is not valid and recovering the shape variation of the waveguide can be seen as a nonlinear inverse scattering problem.

\subsection{General description of the problem}
We describe a 2D varying acoustic waveguide by
\begin{equation} \Omega:=\left\{(x,y)\in \R^2 \ |\ 0<y<h(x)\right\},
\end{equation}
where $h\in \mathcal{C}^2(\R)\cap W^{2,\infty}(\R)$ is a positive profile function defining the top boundary (see an illustration in Figure \ref{6_meas2}). The bottom boundary is assumed to be flat, but a similar analysis could be done when both boundaries vary. In the time-harmonic regime, the wavefield $u_k$ satisfies the Helmholtz equation at frequency $k>0$ with Neumann boundary conditions
\begin{equation}\label{6_eqdebut} 
\left\{\begin{array}{rll} \Delta u_k+k^2u_k &=- f & \text{ in } \Omega,\\ \partial_\nu u_k &=b & \text{ on } \partial\Omega,\end{array}\right. 
\end{equation} 
where $k$ is the frequency, $f$ is an interior source term, and $b$ is a boundary source term. In this work, a waveguide is said to be slowly varying when there exists a small parameter $\eta>0$ such that $\Vert h'\Vert_{\text{L}^\infty(\R)}\leq \eta$ and $\Vert h''\Vert_{\text{L}^\infty(\R)}\leq \eta^2$. We focus here on the recovery of the shape function $h$ modeling the top boundary of the waveguide. Controlled sources $f$ and/or $b$ generate wavefields $u_k$ in $\Omega$ for some frequencies $k>0$. We assume the knowledge of the wavefield measurements $u_k(x,y)$ on the section $\Sigma:=\{x^{\text{meas}}\}\times (0,h(x^{\text{meas}}))\subset\Omega$ for some coordinate $x^{\text{meas}}\in \R$ (see an illustration in Figure \ref{6_meas2}). This inverse problem is good modeling of the monitoring of pipes, optical fibers, or train rails for instance (see \cite{honarvar1,kharrat1,kharrat2}).

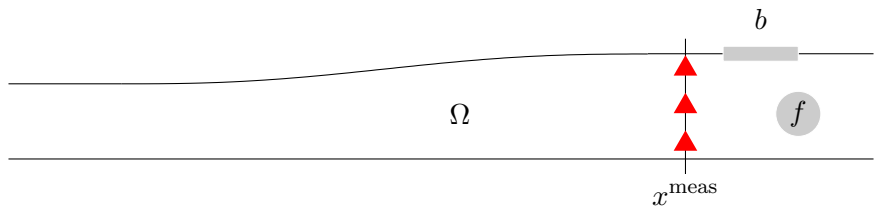
\begin{figure}[h]
\begin{center}
\begin{tikzpicture} 
\draw (-1.5,0) -- (10,0);
\draw (-1.5,1) -- (0,1); 
\draw (7,1.4) -- (10,1.4); 
\draw (4.5,0.6) node{$\Omega$}; 
\draw [domain=0:7, samples=100] plot (\x,{1+12/7/7/7/7/7*\x*\x*\x*(\x*\x/5-7*\x/2+49/3+0.02)});  
\draw (7.5,-0.2)--(7.5,1.6);
\draw (7.5,-0.2) node[below]{$x^{\text{meas}}$};
\draw (7.5,0.2) node[regular polygon,regular polygon sides=3, fill=red, scale=0.5]{};
\draw (7.5,0.7) node[regular polygon,regular polygon sides=3, fill=red, scale=0.5]{};
\draw (7.5,1.2) node[regular polygon,regular polygon sides=3, fill=red, scale=0.5]{};
\draw [white,fill=gray!40] (8,1.3)--(9,1.3)--(9,1.5)--(8,1.5)--(8,1.3); 
\draw (8.5,1.6) node[above]{$b$}; 
\draw [white,fill=gray!40] (9,0.6) circle (0.3);
\draw (9,0.6) node{$f$};
\end{tikzpicture}
\end{center}
\caption{\label{6_meas2} Set up of the inversion in the slowly variable waveguide $\Omega$. An internal source $f$ and/or a boundary source $b$ generate a wavefield propagating in the waveguide and measured on the section $x=x^{\text{meas}}$, represented by red triangles.}
\end{figure}

\subsection{Scientific context}

The detection and reconstruction of shape defects in a waveguide are mentioned in different works. In articles \cite{lu1,abra1,abra2}, the authors use a conformal mapping to map the geometry of the perturbed waveguide to the geometry of a regular waveguide, and suggest to inverse the mapping to recover the width defects. Different linear inversion methods based on the scattered field analysis are also developed in \cite{norgren1,bourgeois1,dediu1}. All these works perform the reconstruction using a single propagation frequency to perform detections and reconstructions of defects. 

Our work concerns a different approach, also used in \cite{bao1,bao2}, where we assume that the data is available for a whole interval of frequencies. This usually provides additional information that should help localize and reconstruct the defect. Moreover, the use of multi-frequency data often provides uniqueness of the reconstruction (see \cite{acosta1}) and better stability (see \cite{bao3,isakov1,sini1}).

The article \cite{bonnetier1} presents a linear method to recover small width variations using back scattering data. However, this work avoids all the locally resonant frequencies of the waveguide. In \cite{niclas1}, we showed that these locally resonant frequencies could be used to recover slowly varying width defects given surface measurements. Given the high sensitivity of the reconstructions obtained using this method, we choose to use the same approach and work only with locally resonant frequencies in this article. 

If $k$ is chosen such that $k=N\pi/h(x^\star_k)$ for some $N\in \N$ and $x^\star_k\in \R$, $k$ is called a locally resonant frequency, $x^\star_k$ is called a resonant position and the Helmholtz problem is not well posed in general (see \cite{bourgeois1}). Nevertheless, it is proved in \cite{bonnetier2} that a unique solution exits as long as the waveguide is slowly varying. The same work provides a suitable approximation of the wavefield that explicitly depends on $x^\star_k$. This approximation is used in \cite{niclas1} with surface measurements to recover the position of $x^\star_k$, and then to reconstruct the shape function $h$. We now aim to use this analysis to solve the inverse problem from section measurements: 
\begin{equation}
\text{Find } \quad h\quad  \text{ from } \quad u_k(x^{\text{meas}},y) \quad \forall y\in(0,h(x^{\text{meas}})),\quad\forall k\in \R_+ \text{ where } k \text{ is locally resonant}.
\end{equation}

Our reconstruction method is inspired by the ideas presented in \cite{bonnetier2} and use wavefield measurements of $u_k$ to recover the position $x^\star_k$. Since $h(x^\star_k)=N\pi/k$, it gives up the information about the waveguide width at this precise location. By taking different locally resonant frequencies $k$ and finding corresponding resonant location $x^\star_k$, we obtain a complete approximation of the width $h$. 

However, the realization strongly differs. Indeed, section measurements contain much less information than surface measurements, and we need to recover the localization of $x^\star_k$ from the value of the resonant mode in one point instead of a whole interval. To circumvent this difficulty, we implement in this paper a layer stripping approach inspired by \cite{sylvester1,somersalo1}. 

\subsection{Main steps of the reconstruction method}

The main steps of the proposed reconstruction method are summarized here: 
\begin{enumerate}
\item From section measurements $ u_k(x^{\text{meas}},y)$, we select frequencies $k$ that makes a chosen mode $N\in\N$ resonant and we filter the data (orthogonal projection) to focus only on the $N$-th mode.
\item We prove that this data is approached by the function~$\Phi\circ \zeta(k)$ where 
\begin{equation}\label{phizeta}
\Phi(x):=\sin(x+\pi/4)\exp(ix+i\pi/4), \qquad \zeta(k):=\int_{x^\star_{k}}^{x^{\text{meas}}} \sqrt{k^2-\frac{N^2\pi^2}{h(x)^2}}\dd x.
\end{equation}
In this function, the dependency on the resonant point $x_k^\star$ and on the shape $h$ is explicit. This point is detailed in Section 3.1 and Corollary \ref{6_coro}. 
\item There is a left inverse function $\Phi^{-1}$ modulus $\pi$. However, to have access to $\zeta(k)$, we need to eliminate the modulus. Using the fact that $k\mapsto \zeta(k)$ is increasing, we prove in Section 3.2 that we can recover an approximation of $\zeta(k)$ on a discrete set of frequencies $k_1,\dots,k_I$ if the discretization step is small enough.
\item From $\zeta(k_1),\dots,\zeta(k_I)$, we develop a layer stripping approach to recover each corresponding resonant $x^\star_{k_i}$ using the fact that $h(x^\star_{k_i})=N\pi/k_i$ (see a scheme in Figure \ref{layer_stripping}) and we perform a triangular linear recovery of each $x^\star_{k_i}$ as explained in Proposition \ref{6_mMV}. 
\item Finally, using the recovery of each $x^\star_{k_i}$ and the fact that $h(x^\star_{k_i})=N\pi/k_i$, we provide an approximation of $h$ whose error is quantified in Theorem \ref{6_th2}. 
\end{enumerate}

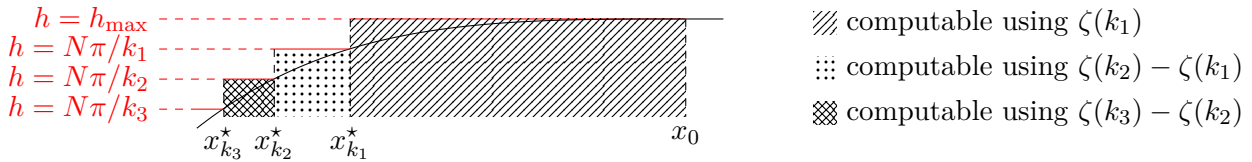
\begin{figure}[h]\begin{center}
\begin{tikzpicture}
\draw (0,1)--(0.5,1);
\draw[dashed] (0,1)--(0,-0.3) node[below]{$x_0$};  
\draw[red] (-4.46,1)--(0,1); 
\draw[dashed] (-4.46,1)--(-4.46,-0.3) node[below]{$x^{\star}_{k_1}$}; 
\fill[color=gray!20, pattern=north east lines] (0,1)--(0,-0.3)--(-4.46,-0.3)--(-4.46,1)--(0,1); 

\fill[color=gray!20, pattern=dots] (-5.47,0.6)--(-5.47,-0.3)--(-4.46,-0.3)--(-4.46,0.6)--(-5.47,0.6); 
\fill[color=gray!20, pattern=crosshatch] (-5.47,0.2)--(-5.47,-0.3)--(-6.14,-0.3)--(-6.14,0.2)--(-5.47,0.2); 
\draw[red] (-7,1) node[left]{$h=h_{\max}$};
\draw[dashed,red] (-7,1)-- (-4.46,1); 
\draw [domain=-6.5:0, samples=100] plot (\x,{1+2*12/7/7/7/7/7*\x/2*\x/2*\x/2*(\x/2*\x/2/5-7*\x/2/2+49/3+0.02)});
\draw[red] (-5.47,0.6)--(-4.46,0.6); 
\draw[dashed] (-5.47,0.6)--(-5.47,-0.3) node[below]{$x^{\star}_{k_2}$}; 
\draw[red] (-7,0.6) node[left]{$h=N\pi/k_1$};
\draw[dashed,red] (-7,0.6)--(-5.47,0.6); 
\draw[red] (-6.14,0.2)--(-5.47,0.2); 
\draw[dashed] (-6.14,0.2)--(-6.14,-0.3);
\draw (-6.14,-0.3) node[below]{$x^{\star}_{k_3}$}; 
\draw[red] (-6.5,-0.2)--(-6.14,-0.2); 
\draw[dashed,red] (-7,0.2)--(-6.14,0.2); 
\draw[red] (-7,0.2) node[left]{$h=N\pi/k_2$};
\draw[dashed,red] (-7,-0.2)--(-6.5,-0.2); 
\draw[red] (-7,-0.2) node[left]{$h=N\pi/k_3$};
\fill[color=gray!20, pattern=north east lines] (2,0.8)--(2,1.1)--(1.7,1.1)--(1.7,0.8)--(2,0.8); 
\fill[color=gray!20, pattern=dots] (2,0.2)--(2,0.5)--(1.7,0.5)--(1.7,0.2)--(2,0.2); 
\fill[color=gray!20, pattern=crosshatch] (2,-0.4)--(2,-0.1)--(1.7,-0.1)--(1.7,-0.4)--(2,-0.4); 
\draw (2,0.95) node[right]{computable using $\zeta(k_1)$};
\draw (2,0.35) node[right]{computable using $\zeta(k_2)-\zeta(k_1)$};
\draw (2,-0.25) node[right]{computable using $\zeta(k_3)-\zeta(k_2)$};
\end{tikzpicture}\end{center}
\caption{\label{layer_stripping} Scheme of the layer stripping approach to recover an approximation of $h$. At each step, we approach the local width by $N\pi/k_i$ and we look for the optimal value $x^{\star}_{k_{i+1}}$ to fit the measurements of $\zeta(k_i)$. }
\end{figure}

The paper is organized as follows. In section 2, we briefly recall the needed results on the modal decomposition and the study of the forward problem. In section 3, we study the inverse problem with measurements taken at the section of the waveguide and provide a stability result for the reconstruction. Finally, in section 4, we illustrate our method with various numerical reconstructions.

\subsection{Notations}

The varying waveguide is denoted by $\Omega$, its boundary by $\partial\Omega$ and the subscript “top” (resp. “bot”) indicates the upper boundary of the waveguide (resp. lower). We denote $\nu$ the outer normal unit vector. For every $r>0$, we set $\Omega_r= \{(x,y)\in \Omega\, |\, |x|<r\}$. Spaces $\tH^1$, $\tH^2$, $W^{1,1}$, $\tH^{1/2}$ over $\Omega$ or $\R$ are classic Sobolev spaces. The Airy function of the first kind (resp. second kind) is denoted by $\Ai$ (resp. $\Bi$). They are linear independent solutions of the Airy equation $y''-xy=0$ (see \cite{abramowitz1} for more results about Airy functions). See in Figure \ref{airy} the graph of these two functions. The term $\delta_{x=s}$ denotes the Dirac distribution at the point $s\in \R$ and the function $\textbf{1}_{E}$ is the characteristic function of the set $E$. Finally, the notation $\{a:b:I\}$ designates the uniform discretization of the interval $[a,b]$ with $I$ points. 

\begin{figure}[h!]
\begin{center}
\scalebox{.5}{\input{airy}}
\caption{\label{airy} Representation of the Airy functions $\Ai$ and $\Bi$.}
\end{center}
\end{figure}

\section{Forward problem and locally resonant modal data}

Before studying the inverse problem associated with the reconstruction of the width in a varying waveguide, we recall some needed tools to study the forward problem. These results where already presented in \cite{niclas1} and their proofs can be found in \cite{bonnetier1,bourgeois1}. However, Theorem \ref{6_th1} is slightly different from the one stated in \cite{niclas1} due to the change of measurements area. 

A useful tool when working in waveguides is the modal decomposition. The following definition provides a modal decomposition in varying waveguides:
\begin{defi}\label{6_def:modes} We define the sequence of functions $(\ph_n)_{n\in\N}$ by
\begin{equation}\label{6_phin}
\forall (x,y)\in \Omega,\quad \ph_n(x,y) :=
\left\{\begin{array}{cl}
1/\sqrt{h(x)}\quad  &\text{if } n=0, \\ 
 \frac{\sqrt{2}}{\sqrt{h(x)}}\cos\left(\frac{n\pi y}{h(x)}\right)\quad &\text{if } n\geq 1,
\end{array}\right.
\end{equation}
which for any fixed $x\in \R$ defines an orthonormal basis of $\tL^2(0,h(x))$.\end{defi}

Hence, a solution $u_k \in \tH^2_{\loc}\big(\Omega\big)$ of \eqref{6_eqdebut} admits a unique modal decomposition
\begin{equation}\label{6_decmode}
 u_k(x,y)=\sum_{n\in \N} u_{k,N}(x) \ph_n(x,y)\quad\text{where}\quad  u_{k,N}(x):=\int_0^{h(x)} u_k(x,y)\ph_n(x,y)\dd y.
\end{equation}
Note that $ u_{k,N}$ does not satisfy in general any nice equation. However, when $h$ is constant (outside of $\supp(h')$), it satisfies an equation of the form $u_{k,N}''+k_n^2 u_{k,N}=- g_n$ where $k_n^2=k^2-n^2\pi^2/h^2$ is the wavenumber. When $h$ is variable, the decomposition \eqref{6_decmode} motivates the following definition:

\begin{defi}\label{6_def:wavenumber}  The local wavenumber function of the mode $n\in\N$ is the complex function $k_n:\R\to \C$ defined by
\begin{equation}
k_n^2(x):=k^2-\frac{n^2\pi^2}{h(x)^2},
\end{equation}
with $\text{Re}(k_n), \text{Im}(k_n)\geq 0$. 
\end{defi}

In this work, as $h(x)$ is non constant, $k_n(x)$ may vanish for some $x\in\R$ and change from a positive real number to a purely imaginary number. We then distinguish three different situations: 

\begin{defi}  A mode $n\in\N$ falls in one of these three situations:
\begin{enumerate}
\item If $n<kh(x)/\pi$ for all $x\in\R$ then $k_n(x)\in(0,+\infty)$ for all $x\in\R$ and the mode $n$ is called propagative. 
\item If $n>kh(x)/\pi$ for all $x\in\R$ then $k_n(x)\in i(0,+\infty)$ for all $x\in\R$ and the mode $n$ is called evanescent. 
\item If there exists $x^\star_k\in \R$ such that $n=kh(x^\star_k)/\pi$ the mode $n$ is called locally resonant. Such points $x^\star_k$ are called locally resonant points, and there are simple if $h'(x^\star_k)\neq 0$, and multiple otherwise. 
\end{enumerate}
A frequency $k>0$ for which there exists at least a locally resonant mode is called a locally resonant frequency.  
\end{defi}

Using the wavenumber function, one can adapt the classic Sommerfeld (or outgoing) condition, defined in \cite{bonnetier1} for regular waveguides, to general varying waveguides $\Omega$. This condition is used to guarantee uniqueness for the source problem given in equation \eqref{6_eqdebut}.

\begin{defi}\label{6_def:outgoing}  A wavefield $ u_k \in \tH^2_{\loc}\big( \Omega\big)$ is said to be outgoing if it satisfies 
\begin{equation} \label{6_sommer}\left|  u_{k,N}'(x)\frac{x}{|x|}-ik_n(x) u_{k,N}(x) \right| \underset{|x|\rightarrow +\infty}{\longrightarrow} 0 \qquad \forall n\in \N,
\end{equation}
where $u_{k,N}$ is given in \eqref{6_decmode}. 
\end{defi}

In all this work, we make the following assumptions:
\begin{asum}\label{6_def:slow} We assume that $h\in \mathcal{C}^2(\R)\cap W^{2,\infty}(\R)$ with $h'$ compactly supported and that 
\begin{equation}\nonumber
\forall x\in \R \quad h_{\min} \leq h(x) \leq h_{\max} \quad \text{ for some }\quad  0<h_{\max}<h_{\min}<\infty.
\end{equation}
We also assume that $h(x)=h_{\min}$ or $h(x)=h_{\max}$ if $x\notin \supp(h')$. For such a function we define a parameter $\eta>0$ that satisfies
\begin{equation}\nonumber
\Vert h'\Vert_{\text{L}^\infty(\R)} <\eta\quad\text{ and } \quad \Vert h''\Vert_{\text{L}^\infty(\R)}<\eta^2.
\end{equation}
\end{asum}
%
%

The forward source problem is defined for every frequency by 
\begin{equation}\label{6_eqmatlab} (\mathcal{H}_k) : \quad 
\left\{\begin{array}{cl} \Delta u_k +k^2 u_k= -f  & \text{ in } \Omega, \\
\partial_\nu u_k =b_\top &\text{ in } \partial\Omega_\text{top},\\
\partial_\nu u_k =b_\bot &\text{ in } \partial\Omega_\text{bot}, \\
u_k\text{ is outgoing.} \end{array}\right. 
\end{equation}
As explained in \cite{bourgeois1}, this problem is not well-posed when the set $\{x \,|\, k_n(x)=0\}$ is a non-trivial interval of $\R$. This especially happens when $k=n\pi/h_{\min}$ or $k=n\pi/h_{\max}$. We then avoid these two situations and we set
\begin{equation}\label{6_delta}
\delta(k):=\min_{n\in \N}\left(\sqrt{\left|k^2-\frac{n^2\pi^2}{{h_{\min}}^2}\right|},\sqrt{\left|k^2-\frac{n^2\pi^2}{{h_{\max}}^2}\right|} \right)>0.
\end{equation}
From now on, we define $(f_n)_{n\in \N}$ the modal decomposition of $f$, and
\begin{equation}\label{6_gn}
g_n(x)=\frac{f_n(x)}{\sqrt{h(x)}}+\varphi_n(1)b_\top(x)\frac{\sqrt{1+(h'(x))^2}}{\sqrt{h}}+\varphi_n(0)b_\bot(x) \frac{1}{\sqrt{h}}. 
\end{equation}
Using the work done in \cite{bonnetier2}, we are able to provide an approximation of the solution of \eqref{6_eqmatlab}. If $h$ is increasing, we can state the following result using Theorem 1 and Remark 5 in \cite{bonnetier2}:

\begin{theorem}\label{6_th1}
Let $k>0$ and let  $h$ be an increasing function defining a varying waveguide $\Omega$ that satisfies Assumption \ref{6_def:slow} with a variation parameter $\eta>0$. Consider sources $f\in \text{L}^2(\Omega)\cap \tL^\infty_c(\Omega)$, $b:=(b_\bot,b_\top)\in ({\text{H}}^{1/2}(\R))^2\cap (\text{L}^\infty_c(\R)^2)$. Assume that there is a unique locally resonant mode $N\in \N$, associated with a simple resonant point $x^\star_k\in\R$.

There exists $\eta_0>0$ depending only on $h_{\min}$, $h_{\max}$, $\delta(k)$, defined in Assumption \ref{6_def:slow} such that if $\eta\leq \eta_0$, then the problem $(\cH_k)$ admits a unique solution $u_k\in \text{H}^2_{\text{loc}}\big(\Omega)$. Moreover, this solution is approached by $u_k^{\text{app}}$ defined for almost every $(x,y)\in \Omega$ by 
\begin{equation}\label{6_greentot}
u_k^{\text{app}}(x,y):=\sum_{n\in \N} u_{k,n}^{\text{app}}(x) \varphi_n(y), \qquad  u_{k,n}^{\text{app}}(x):=\int_\R G_{n}^{\text{app}}(x,s)g_n(s)\dd s,
\end{equation}
where $g_n$ is defined in \eqref{6_gn}, $\ph_n$ is defined in \eqref{6_phin} and $G_{n}^{\text{app}}$ is given by 
\begin{equation}\label{6_greenfunction}
G_{n}^{\text{app}}(x,s):=
\left\{\begin{aligned}
&\frac{i}{2\sqrt{k_n(s)k_n(x)}}\exp\left(i\left|\int_s^xk_n\right|\right), & \quad\text{ if } n<N,\\
&\frac{1}{2\sqrt{|k_n|(s)|k_n|(x)}}\exp\left(-\left|\int_s^x|k_n|\right|\right), & \quad\text{if } n>N,\\
&\left\{\begin{aligned}
\frac{\pi(\xi(s)\xi(x))^{1/4}}{\sqrt{k_n(s)k_n(x)}}\big(i\Ai+\Bi\big)\circ\xi(s)\Ai\circ\xi(x)& \quad\text{ if } x<s, \\
\frac{\pi(\xi(s)\xi(x))^{1/4}}{\sqrt{k_n(s)k_n(x)}}\big(i\Ai+\Bi\big)\circ\xi(x)\Ai\circ\xi(s)& \quad\text{ if } x>s, \\
\end{aligned}\right.
&\quad\text{ if } n=N.
\end{aligned}\right.\end{equation}
Function $k_n$ is the wavenumber function defined in Definition \ref{6_def:wavenumber} and the function $\xi$ is given by 
\begin{equation}\label{6_eq:xi}
\xi(x):=
\left\{\begin{aligned}
\left(-\frac{3}{2}i\int_x^{x^\star_k}k_N(t)\dd t\right)^{2/3} & \text{ if } x<x^\star_k, \\ -\left(\frac{3}{2}\int_{x^\star_k}^x k_N(t) \dd t \right)^{2/3} & \text{ if } x>x^\star_k.
\end{aligned}\right.\end{equation}
Precisely, given a coordinate $x^{\text{meas}}\in \R$, there exist a control of $|u_{k,n}(x^{\text{meas}})-u_{k,n}^\app(x^{\text{meas}})|$ by~$\eta$ for all $n\in \N$. Particularity, for $n=N$, there exists a constant $C_1>0$ depending only on $h_{\min}$, $h_{\max}$, and $N$ such that 
\begin{equation}
|u_{k,N}(x^{\text{meas}})-u_{k,N}^\app(x^{\text{meas}})|\leq \eta C_1 \delta(k)^{-8}\left(\Vert f\Vert_{\text{L}^2(\Omega)}+\Vert b\Vert_{\left({\text{H}}^{1/2}(\R)\right)^2}\right).
\end{equation}

\end{theorem}

This result provides an approximation of the measurements of the $N$-th mode for every frequency, and a control of the approximation error. We notice that at locally resonant frequencies, the wavefield strongly depends on the position of $x^\star_k$, which justifies the idea of using it to develop an inverse method to reconstruct the width $h$. In this work, we assume that we have access to the exact data 
\begin{equation}
u_k(x^{\text{meas}},y) \qquad \forall k>0 \quad \forall y\in  (0,h(x^{\text{meas}})) .
\end{equation}
An illustration is provided in Figure \ref{6_direct} with a representation of the wavefield $u_k$ and the section measurements $u_k(x^{\text{meas}},y)$ when $k$ is a locally resonant frequency. From this general measurement, we define the locally resonant modal data as follows.

\begin{defi} For any $N\in\N$ we denote by $K_N^\text{res}$ the interval of frequencies $k$ such that the mode $N$ is locally resonant at frequency $k$:
\begin{equation}
K_N^\text{res}:=\{k>0\ |\ \exists x_k^\star\in\R,\ N\pi=kh(x_k^\star)\}.
\end{equation}
Then, the locally resonant modal data of the mode $N$ is given by the $N$-th mode measured at position $x^{\text{meas}}$ for all frequencies that make it locally resonant:
\begin{equation}
u_{k,N}(x^{\text{meas}}):=\int_0^{h(x^{\text{meas}})}u_{k}(x^{\text{meas}},y)\ph_n(x^{\text{meas}},y)\dd y,\quad \forall k\in K_N^\text{res}.
\end{equation}
\end{defi}

\begin{figure}[H]
\begin{center}
\input{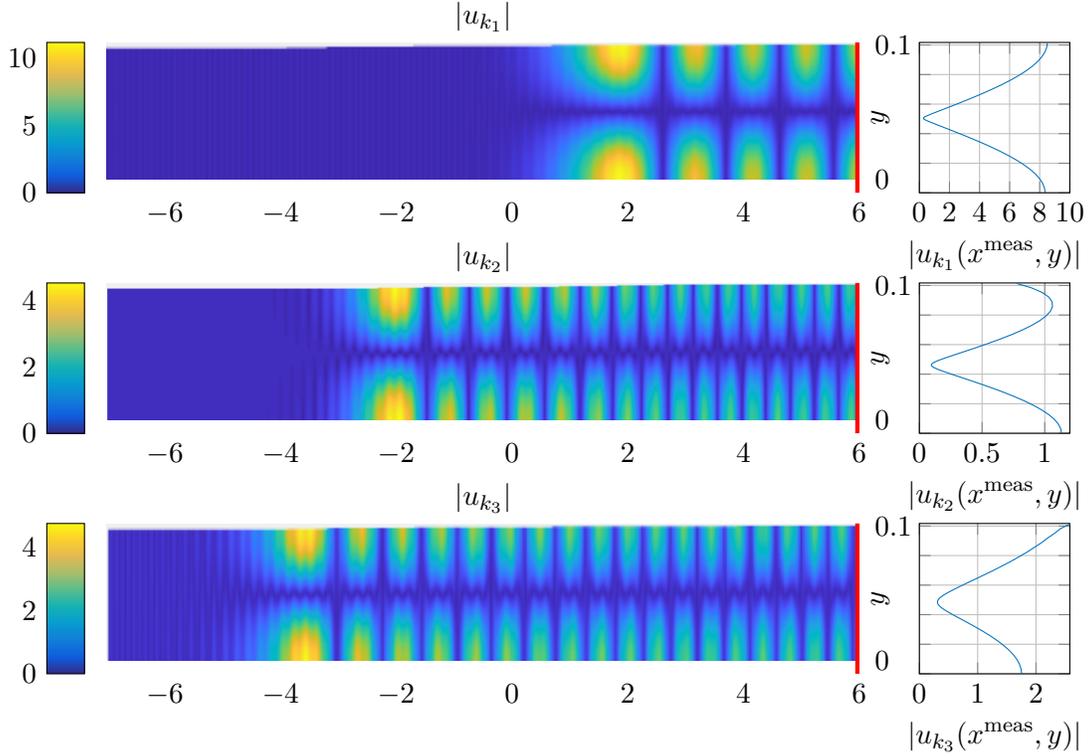}
\end{center}
\caption{\label{6_direct} Illustration of the wavefield measurements in a varying waveguide. The wavefield $u_k$ is solution of \eqref{6_eqmatlab} where $h$ is defined in \eqref{6_h4}, sources are defined in \eqref{6_sourceincr}, $N=1$ and $k_1=31.2$, $k_2=31.5$, $k_3=31.8$ are locally resonant frequencies. Data are generated using the finite element method described in section \ref{6_num}. Left: amplitude of $|u_k|$ in the whole waveguide $\Omega$. Right: measurements of $|u_k|$ on the section $x^{\text{meas}}=6$ represented in red.}
\end{figure}

\section{Discrete inversion from locally resonant modal data}

We now focus on the inverse problem and aim at finding $h$ using the locally resonant modal data. We assume in this section that the waveguide is defined by a fonction $h$ which is slowly varying and strictly increasing in the interior of $\supp(h')$. We also assume that source terms $f$ and/or $b$ defined in \eqref{6_eqmatlab} are supported at the right side of the measurement position $\{x = x^{\text{meas}}\}$. Such a configuration is represented in Figure \ref{6_meas2}.

For a given mode number $N$, we assume that the modal source term $g_N$ defined in \eqref{6_gn} does not cancel everywhere. We also assume that the locally resonant modal data of the mode $N$ given by $u_{k,N}(x^{\text{meas}})$ is fully available for any $k\in K_N^\text{res}$. Assuming that the constant $h_{\min}$ and $h_{\max}$ are given, we can explicit $K_N^\text{res}$ as the open interval 
\begin{equation}\nonumber
K_N^\text{res}=\left(\frac{N\pi}{h_{\max}}, \frac{N\pi}{h_{\min}}\right).
\end{equation}
We detail in Appendix~A how the constants  $h_{\min}$ and $h_{\max}$ can be estimated. 

We do not known if the continuous inverse problem of finding $h$ knowing $u_{k,N}(x^\text{meas})$ for all frequencies $k\in K^\text{res}_N$ is well-posed with a unique solution, and it is still an open question up to our knowledge Therefore, and because we plan on applying the inverse method on experimental data which are discretized, we choose to focus in this paper on the discretized inverse problem of finding information on $h$ knowing $u_{k,N}(x^\text{meas})$ for some discretized frequencies set $\g k$.

We define now an equi-distributed discretization of the frequency interval $k_1,\dots,k_I\in K_N^\text{res}$ such that $k_{i+1}-k_i=\rho>0$ for all $i=1,\dots,I-1$. By extension we also call $k_0:=N\pi/h_{\max}$ and $k_{I+1}:=N\pi/h_{\min}$. As $K_N^\text{res}$ is an open interval, the distance between $\g k:=(k_1,\dots,k_I)$ and the boundary of $K_N^\text{res}$ is non zero and we name $\delta_{\g k}$ the quantity defined by
\begin{equation}\nonumber
\delta_{\g k}:=\min\left(\sqrt{k_{1}^2-k_0^2},\sqrt{k_{I+1}^2-k_{I}^2}\right)>0.
\end{equation}
We associate to every frequency $k_i$ for $i=1,\dots,I$ the local wave number function 
\begin{equation}
k_{i,N}(x):=\sqrt{k_i^2-\frac{N^2\pi^2}{h(x)^2}},
\end{equation} 
which cancels at a single resonant point named $x_i:=x^\star_{k_i}$. The uniqueness of this resonant point is a consequence of the fact that $h$ is monotonous. In order to use Theorem \ref{6_th1}, each locally resonant point $x_i$ must be simple, i.e. $h'(x_i)\neq 0$ which is given by the strict monotonicity of $h$ in the interior of $\supp (h')$. It also implies that  
\begin{equation}
\theta:=\inf_{x\in [x_I,x_1]} \frac{h'(x)}{\eta}>0.
\end{equation}
Most of these quantities are represented in Figure \ref{notations}.

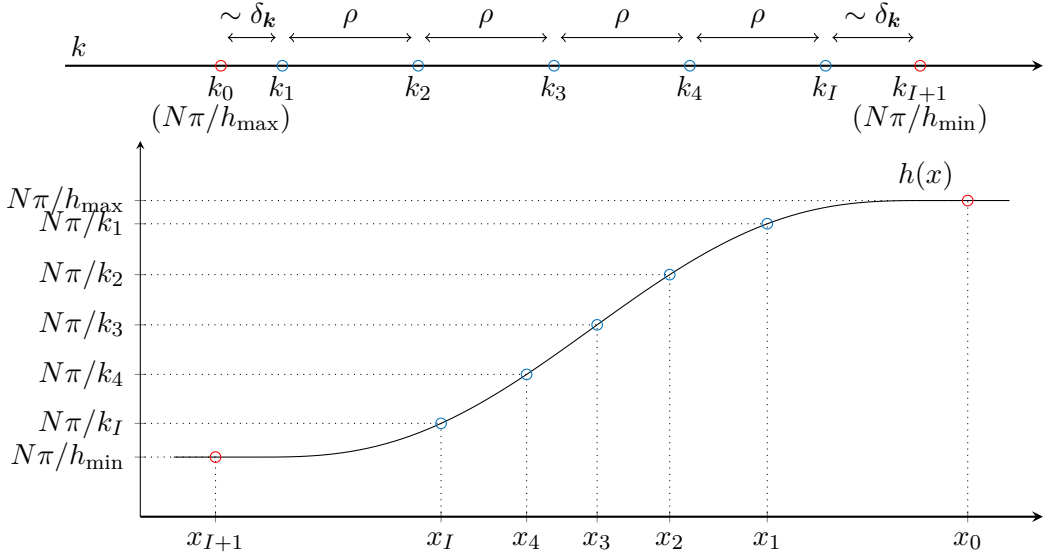
\begin{figure}[h]\center
%
%
\definecolor{mycolor1}{rgb}{0.00000,0.44700,0.74100}%
\begin{tikzpicture}

\begin{axis}[%
width=13cm,
at={(-1cm,0cm)},
height=3cm,
scale only axis,
xmin=30.8,
xmax=32,
axis x line=middle,
   x axis line style={thick},
   enlarge x limits,
   axis y line=left,
   y axis line style={opacity=0},ytick=\empty,
ymin=-2,
ticks=none,
ymax=2,
axis background/.style={fill=white},
legend style={legend cell align=left, align=left, draw=white!15!black}
]
\addplot [color=mycolor1, draw=none, mark=o, mark options={solid, mycolor1}]
  table[row sep=crcr]{%
31	0\\
31.2	0\\
31.4	0\\
31.6	0\\
31.8	0\\
};

\addplot [color=red, draw=none, mark=o, mark options={solid, red}]
  table[row sep=crcr]{%
30.9095052026576	0\\
};

\addplot [color=red, draw=none, mark=o, mark options={solid, red}]
  table[row sep=crcr]{%
31.9392186949968	0\\
};

\draw (30.7,0) node[above]{$k$}; 
\draw[<->] (30.9195052026576,0.5)--(30.99,0.5);
\draw[<->] (31.01,0.5)-- (31.19,0.5);
\draw[<->](31.21,0.5)--(31.39,0.5); 
\draw[<->] (31.41,0.5) --(31.59,0.5); 
\draw[<->] (31.61,0.5) --(31.79,0.5); 
\draw[<->] (31.81,0.5)--(31.9292186949968,0.5); 
\draw (30.9095052026576,0) node[below]{$k_0$};
\draw (30.9095052026576,-0.5) node[below]{$(N\pi/h_{\max})$};
\draw (31,0) node[below]{$k_1$}; 
\draw (31.2,0) node[below]{$k_2$}; 
\draw (31.4,0) node[below]{$k_3$}; 
\draw (31.6,0) node[below]{$k_4$}; 
\draw (31.8,0) node[below]{$k_I$}; 
\draw (30.95,0.5) node[above]{$\sim\delta_{\g k}$};
\draw (31.1,0.5) node[above]{$\rho$}; 
\draw (31.7,0.5) node[above]{$\rho$}; 
\draw (31.3,0.5) node[above]{$\rho$}; 
\draw (31.5,0.5) node[above]{$\rho$}; 
\draw (31.87,0.5) node[above]{$\sim\delta_{\g k}$};
\draw (31.9392186949968,0) node[below]{$k_{I+1}$}; 
\draw (31.9392186949968,-0.5) node[below]{$(N\pi/h_{\min})$}; 

\end{axis}

\begin{axis}[%
width=12cm,
height=5cm,
at={(0cm,-4.5cm)},
scale only axis,
x axis line style={thick},
axis x line=bottom,
axis y line=left,
ytick=\empty,
xtick = {2.10185217615068,0.934592494264073,0.0660555654653432,-0.777866514887448,-1.80136151638459,-4.5,4.5},
xticklabels = {$x_1$,$x_2$,$x_3$,$x_4$,$x_I$, $x_{I+1}$,$x_0$},
ytick={0.101341698502897,0.100692072230442,0.100050721451904,0.0994174890376517,
0.0987922218109998,0.1016384, 0.0983616},
yticklabels={$N\pi/k_1$,$N\pi/k_2$, $N\pi/k_3$, $N\pi/k_4$, $N\pi/k_I$, $N\pi/h_{\max}$,$N\pi/h_{\min}$},
enlarge x limits,
enlarge y limits,
xmin=-4.5,
xmax=4.5,
ymin=0.098,
ymax=0.102,
axis background/.style={fill=white},
legend style={legend cell align=left, align=left, draw=white!15!black}
]
\addplot [color=black]
  table[row sep=crcr]{%
-5	0.0983616\\
-4.8989898989899	0.0983616\\
-4.7979797979798	0.0983616\\
-4.6969696969697	0.0983616\\
-4.5959595959596	0.0983616\\
-4.49494949494949	0.0983616\\
-4.39393939393939	0.0983616\\
-4.29292929292929	0.0983616\\
-4.19191919191919	0.0983616\\
-4.09090909090909	0.0983616\\
-3.98989898989899	0.0983616000658342\\
-3.88888888888889	0.0983616859726668\\
-3.78787878787879	0.0983621868093927\\
-3.68686868686869	0.098363451423826\\
-3.58585858585859	0.0983658002723163\\
-3.48484848484848	0.0983695261768548\\
-3.38383838383838	0.0983748950821799\\
-3.28282828282828	0.0983821468128828\\
-3.18181818181818	0.0983914958305133\\
-3.08080808080808	0.0984031319906857\\
-2.97979797979798	0.0984172213001842\\
-2.87878787878788	0.0984339066740687\\
-2.77777777777778	0.0984533086927806\\
-2.67676767676768	0.0984755263592485\\
-2.57575757575758	0.0985006378559939\\
-2.47474747474747	0.0985287013022369\\
-2.37373737373737	0.0985597555110019\\
-2.27272727272727	0.0985938207462232\\
-2.17171717171717	0.0986308994798511\\
-2.07070707070707	0.0986709771489571\\
-1.96969696969697	0.0987140229128401\\
-1.86868686868687	0.0987599904101315\\
-1.76767676767677	0.0988088185159018\\
-1.66666666666667	0.0988604320987654\\
-1.56565656565657	0.098914742777987\\
-1.46464646464646	0.0989716496805867\\
-1.36363636363636	0.0990310401984465\\
-1.26262626262626	0.099092790745415\\
-1.16161616161616	0.0991567675144143\\
-1.06060606060606	0.0992228272345445\\
-0.95959595959596	0.0992908179281903\\
-0.858585858585859	0.0993605796681265\\
-0.757575757575758	0.0994319453346233\\
-0.656565656565657	0.0995047413725527\\
-0.555555555555556	0.0995787885484936\\
-0.454545454545455	0.0996539027078379\\
-0.353535353535354	0.0997298955318959\\
-0.252525252525253	0.0998065752950025\\
-0.151515151515152	0.0998837476216222\\
-0.0505050505050505	0.0999612162434556\\
0.0505050505050505	0.100038783756544\\
0.151515151515152	0.100116252378378\\
0.252525252525253	0.100193424704998\\
0.353535353535354	0.100270104468104\\
0.454545454545455	0.100346097292162\\
0.555555555555556	0.100421211451506\\
0.656565656565657	0.100495258627447\\
0.757575757575758	0.100568054665377\\
0.858585858585859	0.100639420331874\\
0.95959595959596	0.10070918207181\\
1.06060606060606	0.100777172765456\\
1.16161616161616	0.100843232485586\\
1.26262626262626	0.100907209254585\\
1.36363636363636	0.100968959801554\\
1.46464646464646	0.101028350319413\\
1.56565656565657	0.101085257222013\\
1.66666666666667	0.101139567901235\\
1.76767676767677	0.101191181484098\\
1.86868686868687	0.101240009589868\\
1.96969696969697	0.10128597708716\\
2.07070707070707	0.101329022851043\\
2.17171717171717	0.101369100520149\\
2.27272727272727	0.101406179253777\\
2.37373737373737	0.101440244488998\\
2.47474747474747	0.101471298697763\\
2.57575757575758	0.101499362144006\\
2.67676767676768	0.101524473640751\\
2.77777777777778	0.101546691307219\\
2.87878787878788	0.101566093325931\\
2.97979797979798	0.101582778699816\\
3.08080808080808	0.101596868009314\\
3.18181818181818	0.101608504169487\\
3.28282828282828	0.101617853187117\\
3.38383838383838	0.10162510491782\\
3.48484848484848	0.101630473823145\\
3.58585858585859	0.101634199727684\\
3.68686868686869	0.101636548576174\\
3.78787878787879	0.101637813190607\\
3.88888888888889	0.101638314027333\\
3.98989898989899	0.101638399934166\\
4.09090909090909	0.1016384\\
4.19191919191919	0.1016384\\
4.29292929292929	0.1016384\\
4.39393939393939	0.1016384\\
4.49494949494949	0.1016384\\
4.5959595959596	0.1016384\\
4.6969696969697	0.1016384\\
4.7979797979798	0.1016384\\
4.8989898989899	0.1016384\\
5	0.1016384\\
};

\addplot [color=black,dotted]
  table[row sep=crcr]{%
2.10185217615068	0.097\\
2.10185217615068	0.101341698502897\\
};

\addplot [color=black,dotted]
  table[row sep=crcr]{%
0.934592494264073	0.097\\
0.934592494264073	0.100692072230442\\
};

\addplot [color=black,dotted]
  table[row sep=crcr]{%
0.0660555654653432	0.097\\
0.0660555654653432	0.100050721451904\\
};

\addplot [color=black,dotted]
  table[row sep=crcr]{%
-0.777866514887448	0.097\\
-0.777866514887448	0.0994174890376517\\
};

\addplot [color=black,dotted]
  table[row sep=crcr]{%
-1.80136151638459	0.097\\
-1.80136151638459	0.0987922218109998\\
};

\addplot [color=black,dotted]
  table[row sep=crcr]{%
4.5	0.097\\
4.5	0.1016384\\
};

\addplot [color=black,dotted]
  table[row sep=crcr]{%
-4.5	0.097\\
-4.5	0.0983616\\
};

\addplot [color=black,dotted]
  table[row sep=crcr]{%
-5.5	0.0983616\\
-4.5	0.0983616\\
};

\addplot [color=black,dotted]
  table[row sep=crcr]{%
-5.5	0.1016384\\
4.5	0.1016384\\
};

\addplot [color=black,dotted]
  table[row sep=crcr]{%
-5.5	0.101341698502897\\
2.10185217615068	0.101341698502897\\
};

\addplot [color=black,dotted]
  table[row sep=crcr]{%
-5.5	0.100692072230442\\
0.934592494264073	0.100692072230442\\
};

\addplot [color=black,dotted]
  table[row sep=crcr]{%
-5.5	0.100050721451904\\
0.0660555654653432	0.100050721451904\\
};

\addplot [color=black,dotted]
  table[row sep=crcr]{%
-5.5	0.0994174890376517\\
-0.777866514887448	0.0994174890376517\\
};

\addplot [color=black,dotted]
  table[row sep=crcr]{%
-5.5	0.0987922218109998\\
-1.80136151638459	0.0987922218109998\\
};

\addplot [color=black, draw=none, mark=o, mark options={solid, mycolor1}]
  table[row sep=crcr]{%
2.10185217615068	0.101341698502897\\
};

\addplot [color=black, draw=none, mark=o, mark options={solid, mycolor1}]
  table[row sep=crcr]{%
0.934592494264073	0.100692072230442\\
};

\addplot [color=black, draw=none, mark=o, mark options={solid, mycolor1}]
  table[row sep=crcr]{%
0.0660555654653432	0.100050721451904\\
};

\addplot [color=black, draw=none, mark=o, mark options={solid, mycolor1}]
  table[row sep=crcr]{%
-0.777866514887448	0.0994174890376517\\
};

\addplot [color=black, draw=none, mark=o, mark options={solid, mycolor1}]
  table[row sep=crcr]{%
-1.80136151638459	0.0987922218109998\\
};

\addplot [color=black, draw=none, mark=o, mark options={solid, red}]
  table[row sep=crcr]{%
-4.5	0.0983616\\
};

\addplot [color=black, draw=none, mark=o, mark options={solid, red}]
  table[row sep=crcr]{%
4.5	0.1016384\\
};

\draw(4,0.1016384) node[above]{$h(x)$}; 
\end{axis}

\end{tikzpicture}%
\caption{\label{notations}Illustration of the main notations used in Section 3. In red, the external frequencies and the support of $h'$ found using Appendix A. In blue, the set of frequencies $\g k$ containing $I=5$ frequencies and the corresponding locally resonant points $\g x$. }
\end{figure}

The proposed inversion method consists in recovering the locally resonant points $\g x:=(x_1,\dots,x_I)$ for each frequency in $\g k:=(k_1,\dots,k_I)$. The shape function $h$ will be then approached by a continuous piecewise linear function using the formula $h(x_i)=N\pi/k_i$. Hence, the discrete inverse problem reads 

\begin{equation}
\text{Find } \quad \g x:=(x_1,\dots, x_n) \quad \text{ from }\quad \g d^\text{ex}:=(u_{k_1,N}(x^{\text{meas}}),\dots,u_{k_I,N}(x^{\text{meas}})).
\end{equation}
The analysis of this non-linear inverse problem is decomposed in four steps:
\begin{itemize}
\item \textbf{Step 1: (Data approximation)} In Section 3.1, we prove that 
\begin{equation}
u_{k,N}(x^{\text{meas}})=C\, \Phi\circ\zeta(k)+\cO(\eta),
\end{equation}
where $C$ is a known constant,
\begin{equation}
\Phi(x):=\sin(x+\pi/4)\exp(ix+i\pi/4)\quad\text{and} \quad \zeta(k):=\int_{x^\star_{k}}^{x^{\text{meas}}} \sqrt{k^2-\frac{N^2\pi^2}{h(x)^2}}\dd x.
\end{equation}
\item \textbf{Step 2 : (Inversion of $\Phi$)}   In Section 3.2, we inverse $\Phi$ to obtain $\zeta(k)$. The function $\Phi$ admits a left inverse function modulus $\pi$. We prove that if the frequency discretization step $\rho$ is small enough, the global left inverse can be found and applied to the data. It leads to an approximation of  $\zeta(k)$ up to an error in $\cO(\eta)$. 

\item \textbf{Step 3: (Discretization of the integral)} In Section 3.3, we use the numerical integration elementary formula

\begin{equation}\nonumber
\begin{aligned}
\int_{x_{j+1}}^{x_j} k_{i,N}(x)\dd x &\approx (x_j-x_{j+1})\left(\frac14k_{i,N}(x_{j+1})+\frac34k_{i,N}(x_{j})\right)\\
&\approx (x_j-x_{j+1})\left(\frac14\sqrt{k_i^2-k_{j+1}^2}+\frac34\sqrt{k_i^2-k_{j}^2}\right)\\
\end{aligned}
\end{equation}
to approach $\zeta(k_i)$. It allows to run a stripping method to recover the differences $x_j-x_{j+1}$ from the approached values of $\zeta(k_i)$, which provides an approximation of $\g x$. As we will see, this stripping method conduces to triangular linear system.

\item \textbf{Step 4: (Reconstuction of $h$)} In Section 3.4, we use the approximation $\g x$ to provide a reconstruction of $h$. We quantify the reconstruction error of this inverse method as well as its stability. 
\end{itemize}

\subsection{Data approximation}

First, we want to approach the locally resonant modal data  $u_{k,N}(x^{\text{meas}})$. Theorem \ref{6_th1} provides a first approximation $u_{k,N}^\app(x^{\text{meas}})$ of these measurements. However, this expression is hardly usable to do inverse problems. Using the fact that the measurement section $\{x=x^{\text{meas}}\}$ is far from the defect, we can simplify this approximation:

\begin{prop}\label{6_meassec}
Let $k\in \g k$ and assume that $\Omega$ satisfies Assumption \ref{6_def:slow}. Let $x^{\text{meas}}\in \R$, there exist $\eta_0,A>0$ such that if $\eta<\eta_0$ and $x^{\text{meas}}-x_{0}>A$, then there exists $C_2>0$ depending on $\delta_{\g k}$, $h_{\min}$ and $h_{\max}$ such that for every $k\in \delta_{\g k}$, 
\begin{equation}
u_{k,N}(x^{\text{meas}})=q(k)\,\Phi\circ\zeta(k)+\cO(\eta), \qquad \left|u_{k,N}(x^{\text{meas}})-q(k)\,\Phi\circ\zeta(k)\right|\leq \eta C_2,
\end{equation}
where $q(k):=\int_\R g_N(z) e^{ik_{N}(x^{\text{meas}})(z-x^{\text{meas}})}\dd z/k_N(x^{\text{meas}})$ is a constant depending only on the source~terms.
\end{prop}
\begin{proof}
Using Theorem \ref{6_th1}, we know that for every $k\in K$, 
\begin{equation}\nonumber
|u_{k,N}(x^{\text{meas}})-u_{k,N}^{\text{app}}(x^{\text{meas}})|\leq \eta C_1 \delta^{-8}(K) \left(\Vert f \Vert_{\text{L}^2(\R)} +\Vert b\Vert_{\text{H}^{1/2}(\R)}\right),
\end{equation}
where 
\begin{equation}\nonumber
u_{k,N}^{\app}(x^{\text{meas}})=\frac{\pi(-\xi(x^{\text{meas}}))^{1/4}\Ai(\xi(x^{\text{meas}}))}{k_{N}(x^{\text{meas}})} \int_\R  g_N(z) (-\xi(z))^{1/4}(i\Ai+\Bi)(\xi(z))\dd z . 
\end{equation}
Using approximations of Airy functions provided in \cite{olver3} and \cite[Chap. 9.7]{olver4}, we see that there exists a constant $c_1>0$ such that for all $x>0$,
\begin{equation}\nonumber
\left|\sqrt{\pi}x^{1/4}\Ai(-x)+\sin\left(\frac{2}{3}x^{3/2}+\frac{\pi}{4}\right)\right|\leq \frac{c_1}{x^{5/4}}, 
\end{equation}
\begin{equation}\nonumber
\left|\sqrt{\pi}x^{1/4}\Bi(-x)+\cos\left(\frac{2}{3}x^{3/2}+\frac{\pi}{4}\right)\right|\leq \frac{c_1}{x^{5/4}}.
\end{equation}
and so
\begin{equation}\nonumber 
\left|\sqrt{\pi}x^{1/4}(i\Ai+\Bi)(-x)+\exp\left(i\frac{2}{3}x^{3/2}+i\frac{\pi}{4}\right)\right|\leq \frac{2c_1}{x^{5/4}},
\end{equation}
Replacing $x$ by $-\xi(x)$ or $-\xi(x^{\text{meas}})$ leads to 
\begin{equation}\nonumber
\left|u_{k,N}^\app(x^{\text{meas}})-q(k)\,\Phi\circ\zeta(k)\right|\leq \delta^{-1}\Vert g_N\Vert_{\text{L}^1(\R)}\frac{3c_1+2c_1^2}{(-\xi(x^{\text{meas}}))^{5/4}}.
\end{equation}
We conclude the proof by noticing that 
\begin{equation}\nonumber
-\xi(x^{\text{meas}})=\left(\frac{3}{2}\int_{x^\star_k}^{x^{\text{meas}}}k_N(t)\dd t\right)^{2/3} \geq (x^{\text{meas}}-x_0)^{2/3} k_N(x^{\text{meas}})^{2/3}\underset{x^{\text{meas}}\rightarrow +\infty}{\longrightarrow} +\infty.
\end{equation}
\end{proof}

Since source terms are assumed to be chosen, $q$ is an explicit non vanishing quantity and Proposition \ref{6_meassec} can be rewritten as follows: 
\begin{coro}\label{6_coro}
Let $k\in \g k$, and assume that $\Omega$ satisfies Assumption \ref{6_def:slow}. Let $x^{\text{meas}}\in \R$, there exist $\eta_0,A>0$ such that if $\eta<\eta_0$ and $x^{\text{meas}}-x_{0}>A$, then there exists $C_2>0$ depending on $\delta_{\g k}$, $h_{\min}$ and $h_{\max}$ such that for every $k\in \g k$,
\begin{equation}
\Phi\circ\zeta(k)=\frac{1}{q(k)} u_{k,N}(x^{\text{meas}}) +\cO(\eta), \qquad \left|\Phi\circ\zeta(k)-\frac{1}{q(k)} u_{k,N}(x^{\text{meas}})\right|\leq \eta c_2C_2,
\end{equation}
where $c_2:=\max_{k\in \g k} 1/q(k)$.
\end{coro}

\subsection{Left inverse of the function $\Phi$}

Using the previous Corollary, we now have a good approximation of $\Phi\circ\zeta(k)$ for all frequencies $k\in \g k$. To get access to $\zeta(k)$, we need to find the left inverse function of $\Phi$. Since $\Phi$ is $\pi$-periodic, we can only provide a partial left inverse function $\Phi^{-1}$  modulus $\pi$ satisfying $\Phi^{-1}(\Phi(\theta))=\theta \,\,\text{mod}(\pi)$,~with 

\begin{equation}
 \Phi^{-1}(z):=\left\{\begin{array}{cl} \arcsin(|z|) & \text{ if } |z|<1/2 \text{ and } \text{Real}(z)\geq 0,\\
\pi-\arcsin(|z|) & \text{ if } |z|<1/2 \text{ and } \text{Real}(z)<0,\\
\arccos\left(\text{Real}(z)/|z|\right) & \text{ if } |z|\geq 1/2.\end{array}\right. 
\end{equation}

We represent in Figure \ref{6_postpross} the values of $\Phi^{-1}(u_{k,N}(x^{\text{meas}})/q(k))
 \text{ mod}(\pi)$ for different sets of frequency. We can see that $\rho$, the discretization step of $\g k$, needs to be sufficiently small if we expect to get rid of the modulus $\pi$. Using the fact that $\zeta$ is increasing and assuming that $\rho$ is small enough, the following proposition proves that we can get rid of the modulus $\pi$ up to a constant. We set
\begin{equation}
\ell=\left\lfloor\frac{\zeta(k_1)}{\pi}\right\rfloor,\qquad 
t_i^{\app}=\zeta(k_i)-\ell \pi, \quad 1\leq i \leq m.
\end{equation}

\begin{figure}[h]
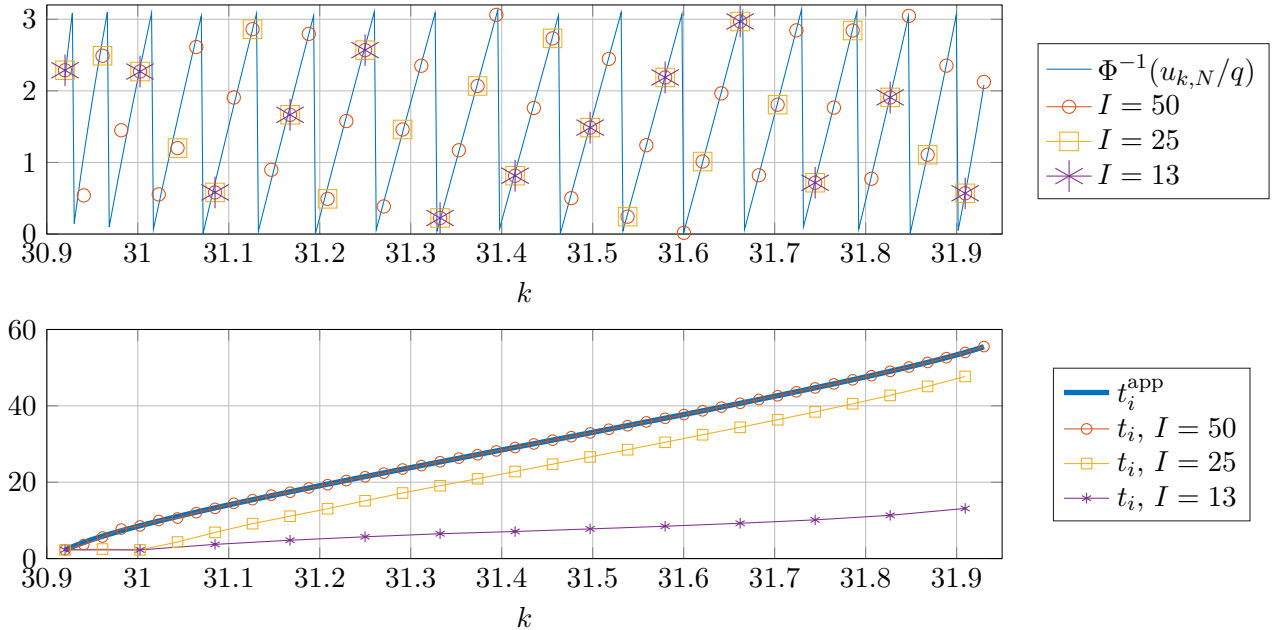

\begin{center}
\begin{tikzpicture}
\input{postpross1}

\input{postpross2}
\end{tikzpicture}
\end{center}
\caption{\label{6_postpross} Computation of the left inverse of $\Phi$ for a set of frequencies $\g k=\{30.92:31.93:I\}$, $x^{\text{meas}}=6$, $h$ defined in \eqref{6_h1} and source defined in \eqref{6_sourceincr}. On the top, representation of $\Phi^{-1}(v_k^\app)$ and $\Phi^{-1}(v_k)$ modulus~$\pi$. On the bottom, representation of $t_i^\app$ and $t_i$ after getting rid of the modulus~$\pi$. Here, $\ell^\app=0$. }
\end{figure}

\begin{prop} \label{6_ti}
Assume that $\Omega$ satisfies Assumptions \ref{6_def:slow}. There exist $\eta_0,A>0$ such that if $\eta<\eta_0$, $x^{\text{meas}}-x_{0}>A$, there exist $\rho_0>0$ depending on $x^{\text{meas}}$ and $\eta$ such that if $\rho<\rho_0$, then there exist a unique sequence $(t_i)_{1\leq i \leq m}$ such that 
\begin{equation}
\forall 1\leq i \leq m \quad t_i=\Phi^{-1}\left(\frac{u_{k,N}(x^{\text{meas}})}{q(k)}\right) \,\, \text{mod}(\pi), \qquad  |t_{i+1}-t_i|< \pi/2.
\end{equation}
Moreover, there exists a constant $C_3>0$ depending on $\delta_{\g k}$, $\theta$, $h_{\min}$ and $h_{\max}$ such that 
\begin{equation}
\forall	1\leq i \leq m \quad |t_i-t_i^{\app}|\leq C_3\eta. 
\end{equation}
\end{prop}

\begin{proof}
Using Corollary \ref{6_coro}, there exist $\eta_0,A>0$ such that if $\eta<\eta_0$ and $x^{\text{meas}}-x_0>A$ then
\begin{equation}\nonumber
\left|\frac{u_{k,N}(x^\text{meas})}{q(k)}-\Phi\circ \zeta(k)\right|\leq \min(1/4,\eta c_2C_2).
\end{equation} 
The quantity $1/4$ is chosen to ensure that in each case of the definition of $\Phi^{-1}$, we then have $1/4\leq z\leq 3/4$. Then, we see that for every $k\in K$, 
\begin{equation}\nonumber
\left|\Phi^{-1}\left(\frac{u_{k,N}(x^\text{meas})}{q(k)}\right)-\Phi^{-1}(\Phi\circ \zeta(k))\right|\leq \frac{1}{\sqrt{1-(3/4)^2}}\left|\frac{u_{k,N}(x^\text{meas})}{q(k)}-\Phi\circ \zeta(k)\right|< \min(1.6\,\eta c_2 C_2,\pi/8).
\end{equation}
We also notice that 
\begin{equation}\nonumber
0\leq t^\app_{i+1}-t^\app_i=\int_{x_{i+1}}^{x_i} k_{i+1,N} +\int_{x_i}^{x^{\text{meas}}} (k_{i+1,N}-k_{i,N}).
\end{equation}
Using the fact that
\begin{equation}\label{distxi}
|x_i-x_{i+1}|=\left|h^{-1}\left(\frac{N\pi}{k_i}\right)-h^{-1}\left(\frac{N\pi}{k_{i+1}}\right)\right|\leq \frac{1}{\theta\eta}\left|\frac{N\pi}{k_i}-\frac{N\pi}{k_{i+1}}\right|\leq \frac{\rho N \pi}{k_{0}^2\theta\eta},
\end{equation}
we have
\begin{equation}\nonumber
\int_{x_{i+1}}^{x_i} k_{i+1,N}\leq k_{n+1,N}(x^{\text{meas}})(x_i-x_{i+1})=\cO( \rho\eta^{-1}),
\end{equation}
\begin{equation}\nonumber
\int_{x_i}^{x^{\text{meas}}} (k_{i+1,N}-k_{i,N})\leq (x^{\text{meas}}-x_{{m+1}})\rho\frac{k_{m+1}}{\delta_{\g k}}=\cO(\rho (x^{\text{meas}}-x_0)),
\end{equation}
where we denote by $\cO$ every majoration depending only on $\delta_{\g k}$, $\theta$, $h_{\min}$, $h_{\max}$. If $\rho$ is small enough compared to $(x^{\text{meas}}-x_0)^{-1}$ and $\eta$ then
\begin{equation}\nonumber
| t^\app_{i+1}-t^\app_i| <\pi/4.
\end{equation}
Therefore, 
\begin{equation}\nonumber
|t_{i+1}-t_i|\leq |t_{i+1}-t_{i+1}^{\app}|+|t^\app_{i+1}-t^\app_i|+|t_{i}^\app-t_i|<\pi/2.
\end{equation}
Since there is a one-to-one correspondence between a sequence where the distance between each term does not exceed $\pi/2$ and its representation modulus $\pi$, we conclude our proof. 

\end{proof}

This result is illustrated in Figure \ref{6_postpross} where we see that providing a sufficiently small discretization of the frequency interval, we can recover an approximation of each $t_i^\app$ giving measurements of~$\Phi^{-1}(v_k)$. 

Lastly, we have to find the constant $\ell\in \N$ in order to reconstruct a complete approximation of $\zeta(k)$. Since 
\begin{equation}
t_1\approx t_1^\app \approx (x^{\text{meas}}-x_{k_1}^\star)k_{1,N}(x^{\text{meas}}) -\ell \pi, \quad t_2\approx t_2^\app \approx (x^{\text{meas}}-x_{k_1}^\star)k_{2,N}(x^{\text{meas}})-\ell \pi,
\end{equation}
we propose to approach $\ell$ by the following expression: 
\begin{equation}\label{6_ellapp}
\ell^{\app}=\left\lfloor\frac{t_2k_{1,N}(x^{\text{meas}})-t_1k_{2,N}(x^{\text{meas}})}{\pi(k_{2,N}(x^{\text{meas}})-k_{1,N}(x^{\text{meas}}))}\right\rfloor.
\end{equation} 
The following lemma proves that under certain assumptions, $\ell^{\app}$ is equal to $\ell$:

\begin{lem}\label{6_ell}
Under the same assumptions as in Proposition \ref{6_ti}, if $\delta_{\g k}$, $x_0-x_{k_2}^\star$ and $\eta$ are sufficiently small then $\ell=\ell^{\app}$.
\end{lem}

\begin{proof}
We notice that
\begin{equation}\nonumber
\ell\pi=\int_{x_1}^{x^{\text{meas}}} k_{1,N}(z) \dd z -t_1^{\app}\approx (x^{\text{meas}}-x_1)k_{1,N}(x^{\text{meas}})-t_1.
\end{equation}
More precisely, we introduce $\eps_1$ such that 
\begin{equation}\nonumber
(x^{\text{meas}}-x_1)k_{1,N}(x^{\text{meas}})=\ell\pi+t_1+\eps_1,
\end{equation}
and using Proposition \ref{6_ti}, there exists a constant $c_3>0$ such that 
\begin{equation}\nonumber
|\eps_1|\leq C_3 \eta +\left|\int_{x_1}^{x_0}(k_{1,N}(z)-k_{1,N}(x_0))\dd z\right| \leq  C_3\eta+c_3(x_0-x_2)\delta_{\g k}.
\end{equation}
Similarly, we introduce $\eps_2$ such that $(x^{\text{meas}}-x_1)k_{2,N}(x^{\text{meas}})=\ell\pi+t_2+\eps_2$ and 
\begin{equation}\nonumber
 |\eps_2| \leq C_3\eta+d_3(x_0-x_2)\delta_{\g k}.
\end{equation}
It follows that 
\begin{equation}\nonumber
|\ell- \ell^{\app}|\leq \frac{|\eps_2|k_{2,N}(x^{\text{meas}})+|\eps_1|k_{1,N}(x^{\text{meas}})}{k_{2,N}(x^{\text{meas}})-k_{1,N}(x^{\text{meas}})},
\end{equation}
and this quantity is smaller than $1$ if $\eps_1$ and $\eps_2$ are small enough. Since $\ell,\ell^\app\in \N$, it concludes the proof.
\end{proof}

Using this lemma, we can find the value of $\ell$. We represent in Figure \ref{6_postpross} the different steps to compute the left inverse of $\Phi$:
\begin{itemize}
\item We apply the function $\Phi^{-1}$ to the approximations of $\Phi\circ\zeta(k)$.
\item We recover each $\zeta(k_i)$ up to a constant $\ell$ is $\rho$ is small enough using the increase of $\zeta$. 
\item We compute the value of the constant $\ell$ using \eqref{6_ellapp}.
\end{itemize} 
As mentioned in Proposition \ref{6_ti}, we clearly see the importance of choosing a sufficiently small discretization of the frequency interval $\delta_{\g k}$ in order to get a good approximation of each $\zeta(k_i)$. 

To summarize, we proved in this section that if the step of discretization of $\g k$ is sufficiently small, we are able to post process the measurements in order to find an approximation of $\zeta(k_i)$ for every frequency $k_i\in \g k$. Moreover, this approximation is improving is $\eta$ diminishes. From now on, we work on data $(d_i)_{1\leq i \leq m}$ satisfying
\begin{equation}
\forall i=1,\dots, m \qquad d_i:=t_i+\ell\pi, \quad \left|d_i-\zeta(k_i) \right| \leq C_3\eta.
\end{equation}

\subsection{Triangular linear recovery of $\g x$}

Using the left inverse of $\Phi$ and the data provided in the previous section, we now try to reconstruct from $(\zeta(k_i))_{1\leq i \leq I}$ the set $\g x$. To do so, we need to find a link between $(\zeta(k_i))_{1\leq i \leq I}$ and $\g x$. Noticing that 
\begin{equation}
\forall 1\leq i\leq I \quad h(x_i)=\frac{N\pi}{k_i}, \qquad \forall 1\leq j<i\leq I\quad k_{i,N}(x_j)=\sqrt{k_i^2-\frac{N^2\pi^2}{h(x_j)^2}}=\sqrt{k_i^2-k_j^2},
\end{equation}
we choose to discretize the integral $\tint_{x_i}^{x^{\text{meas}}}k_{i,N}$ on the grid $(x_j)_{1\leq j \leq i}$. We define the known quantities
\begin{equation}\label{pij}
p_{i,j}:=\sqrt{|k_i^2-k_j^2|}.
\end{equation}

A first idea would be to use a rectangular rule to discretize the integral, as illustrated in the introduction in Figure \ref{layer_stripping}. However, since $h$ is increasing, $x\mapsto k_{i,N}(x)$ is concave and we know that for all $j>2$,
\begin{equation}
\int_{x_j}^{x_{j-1}}k_{i,N}\geq (x_{j-1}-x_j)\frac {k_{i,N}(x_j)+k_{i,N}(x_{j-1})}{2}=(x_{j-1}-x_j)\frac{p_{i,j}+p_{i,j-1}}{2},
\end{equation}
\begin{equation}
\int_{x_j}^{x_{j-1}}k_{i,N}\leq (x_{j-1}-x_j)k_{i,N}(x_{j-1})=(x_{j-1}-x_j)p_{i,j-1}.
\end{equation}
Instead of using a rectangular rule, we choose to approach the integral using the mean of these two bounds:
\begin{equation}
\int_{x_j}^{x_{j-1}}k_{i,N}\approx \frac{1}{4}(x_{j-1}-x_j)\left(p_{i,j}+3p_{i,j-1}\right). 
\end{equation}
When $j=1$, we cannot apply this approach and we simply choose to use
\begin{equation}
\int_{x_1}^{x^{\text{meas}}}k_{i,N}\approx (x^{\text{meas}}-x_1)k_{i,N}(x^{\text{meas}})=(x^{\text{meas}}-x_1)p_{1,0}. 
\end{equation}
\begin{rem}
We chose here to approach the integral by taking the mean of a rectangle and a trapezoidal rule. However, any barycenter between these two bounds would also work. Further investigations may prove that the mean may not be the optimal choice. However, giving the lack of regularity of $k_{i,N}$ around $x_i$, any quadrature method should give the same error estimation up to a constant. 
\end{rem}
We define the triangular matrix
\begin{equation}\label{6_M}
T:=\frac{1}{4}\left(\begin{array}{ccccc} 4p_{1,0} & 0 & 0 & \hdots & 0 \\
4p_{2,0} & 3p_{2,1} & 0 &  & \vdots \\[2mm]
4p_{3,0} & p_{3,2}+3p_{3,1} & \ddots & \ddots & \vdots \\ 
\vdots & \vdots & \ddots & 3p_{I-1,I-2} & 0 \\
4p_{I,0} & p_{I,2}+3p_{I,1} & \hdots & p_{I,I-1}+3p_{I,I-2} & 3p_{I,I-1}\end{array}\right), 
\end{equation}
and we expect that
\begin{equation}
TV\approx d:=(d_i)_{1\leq i \leq I}, \qquad \text{ where }\qquad  V=\left(\begin{array}{c}
x^{\text{meas}}-x_1 \\ x_1-x_2 \\ \vdots \\ x_{I-1}-x_I \end{array}\right). 
\end{equation}
The following proposition quantify the error of this approximation: 

\begin{prop}\label{6_mMV}
Assume that $\Omega$ satisfies Assumption \ref{6_def:slow}. There exist $\eta_0,A>0$ such that if  $\eta<\eta_0$, $x^{\text{meas}}-x_0>A$, there exists $\rho_0>0$ such that if $\rho<\rho_0$, then there exist constants $C_4,C_5>0$ depending on $\delta_{\g k}$, $\theta$, $|\supp(h')|$, $h_{\min}$ and $h_{\max}$ such that 
\begin{equation}
\Vert d-TV\Vert_\infty \leq C_3\eta+C_4(x_0-x_1)+ C_5\eta^{1/2}\rho^{-1}.
\end{equation}
\end{prop}

\begin{rem}
We notice that in addition of the need for $\eta$
to be small, we need to ensure that $x_0$ is sufficiently precise and close to $x_1$ if we want a small error between $d$ and $TV$. This assumption was already evoked in Lemma \ref{6_ell} to find the quantity $\ell$. \end{rem}

\begin{rem}
We notice in this inequality that if $\eta$ tends to zero then the error decreases. However, for a fixed $\eta$, the parameter $\rho$ needs to be chosen wisely since it needs to be small enough in order for the post processing of the measurements to work, and not too small so that the matrix $T$ is well conditioned. We will discuss this point with more detail in Section 4.
\end{rem}

\begin{proof}
Using the proof of Lemma \ref{6_ell}, we already know that 
\begin{equation}\nonumber
\left| \int_{x_1}^{x^{\text{meas}}} k_{i,N} -(x^{\text{meas}}-x_1)p_{i,0}\right| \leq C_3\eta+ c_3(x_0-x_1)\delta_{\g k}.  
\end{equation}
Then, for $x\in (x_j,x_{j-1})$, we denote by $z\in (x_j,x_{j-1})$ the coordinate such that $k_{i,N}(z)=(p_{i,j}+3p_{i,j-1})/4$. If $j<i$, then 
\begin{equation}\nonumber
\left|k_{i,N}(x)-k_{i,N}(z)\right|\leq |x-z|\frac{N^2\pi^2\eta}{h_{\min}^3k_{i,N}(x_j)}\leq|x_j-x_{j-1}|\frac{N^2\pi^2\eta}{h_{\min}^3\sqrt{2}\sqrt{k_{0}}\sqrt{i-j}\rho^{1/2}}.
\end{equation}
If $j=i$ and $x>z$, then  
\begin{equation}\nonumber
\left|k_{i,N}(x)-k_{i,N}(z)\right|\leq |x_i-x_{i-1}|\frac{4N^2\pi^2\eta }{3h_{\min}^3p_{i,i-1}}\leq |x_i-x_{i-1}|\frac{4N^2\pi^2\eta }{3h_{\min}^3\rho^{1/2}\sqrt{2}\sqrt{k_{0}}}.
\end{equation}
If $j=i$ and $x<z$, we use the approximation of $k_{i,N}$ near $x_i$ demonstrated in \cite[(20)]{niclas1} to deduce that there exists a constant $c_4>0$ depending on $\theta$ such that
\begin{equation}\nonumber
k_{i,N}(x)\geq c_4\eta^{1/2}(x-x_i)^{1/2}.
\end{equation}
If follows that 
\begin{equation}\nonumber
|k_{i,N}(x)-k_{i,N}(z)|\leq |x_i-x_{i-1}|\frac{N^2\pi^2\eta^{1/2}}{h_{\min}^3c_4(x-x_i)^{1/2}},
\end{equation}
\begin{equation}\nonumber \int_{x_i}^{z} |k_{i,N}(x)-k_{i,N}(z)|\dd x\leq 2|x_i-x_{i-1}|^{3/2}\eta^{1/2}\frac{N^2\pi^2}{h_{\min}^3c_4}.
\end{equation}
By noticing that $|x_i-x_{i-1}|\leq x_0-x_{I+1}$, we conclude that there exists a constant $c_5>0$ such that 
\begin{equation}\nonumber
\left|\int_{x_j}^{x_{j-1}}k_{i,N}- \frac{1}{4}(x_{j-1}-x_j)\left(p_{i,j}+3p_{i,j-1}\right)\right|\leq c_5\eta^{1/2}\rho^{-1/2}\left(\delta_{i=j}+\frac{1}{\sqrt{i-j}}\delta_{i>j}\right). 
\end{equation}
By assembling everything, we see that
\begin{equation}\nonumber
|d_i-(TV)_i|\leq C_3\eta+ \delta_{\g k} c_3(x_0-x_2)+c_5\eta^{1/2}\rho^{-1/2}\left(1+\sum_{j=1}^{i-1}\frac{1}{\sqrt{j}}\right).
\end{equation}
We conclude the proof by noticing that $\sum_{j=1}^i j^{-1/2}=\cO\left(\sqrt{i}\right)=\cO(\rho^{-1/2})$. 
\end{proof}

The matrix $T$ is a lower triangular matrix with non vanishing diagonal entries so it is invertible. To find an approximation of vector $V$ and of $\g x$ given data $d$, we define
\begin{equation}
V^{\app}:=T^{-1}d, \qquad (\g x^\app)_i:=x^\app_i:=x^{\text{meas}}-\sum_{j=1}^i V^{\app}_j.  
\end{equation}

\subsection{Reconstruction of $h$}

We now have approximations of each $x_i$. Since we know that $h(x_i)=N\pi/k_i$, we define the approximation of the width $h$ by
\begin{equation}\label{6_happ}
 h^{\app}(x^\app_i)=\frac{N\pi}{k_i} \quad \forall\,  k_i\in \g k, \qquad  h^{\app}(x_0)=\frac{N\pi}{k_{m+1}}, \qquad h^{\app}(x_{m+1})=\frac{N\pi}{k_{0}}. 
\end{equation}
The following Theorem concludes the reconstruction process and proves that we are able to quantify the error of reconstruction between $h^{\app}$ and $h$:  

\begin{theorem}\label{6_th2}
Let $\g k$ be a subset of $K^\text{res}_N$. Assume that $\Omega$ satisfies assumption \ref{6_def:slow}. There exist $\eta_0,A>0$ such that if $\eta<\eta_0$ and $x^{\text{meas}}-x_0>A$, there exists $\rho_0>0$ such that if $\rho<\rho_0$ then there exist constants $C_3,C_4,C_5,C_6>0$ depending on $\theta$, $\delta_{\g k}$, $|\supp(h)|$, $h_{\min}$ and $h_{\max}$ such that 
\begin{equation}
\Vert h^{\app}(\g x^\app)-h(\g x^\app)\Vert_\infty\leq \eta\rho^{-5/2} C_6\left(C_3\eta+C_4(x-x_1)+ C_5\eta^{1/2}\rho^{-1}\right).
\end{equation}
\end{theorem}

\begin{rem}
We notice in this inequality that if $\eta$ tends to zero then the error tends to zero. However, as mentionned before, for a fixed $\eta$ the parameter $\rho$ needs to be chosen wisely since $T$ is not well conditioned. We illustrate it in Figure \ref{6_choixrho} where we present two different reconstructions of $h$ with the same data $(d_i)$ discretized for two different sets of frequencies $\g k$. We notice that the relative error of reconstruction $\Vert h-h^\app\Vert_\infty/h_{\max}$ increases when $I$ increases (and when $\rho$ decreases). 
\end{rem}

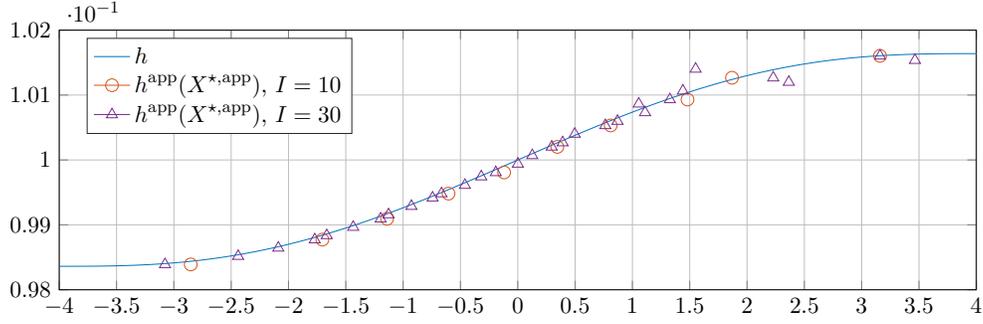
\begin{figure}[h]
\begin{center}
\scalebox{0.8}{
%
%
\definecolor{mycolor1}{rgb}{0.00000,0.44700,0.74100}%
\definecolor{mycolor2}{rgb}{0.85000,0.32500,0.09800}%
\definecolor{mycolor3}{rgb}{0.49400,0.18400,0.55600}%
\begin{tikzpicture}

\begin{axis}[%
width=6in,
height=1.7in,
at={(1.795in,0.489in)},
scale only axis,
xmin=-4,
xmax=4,
ymin=0.098,
ymax=0.102,
grid=major, 
scaled y ticks=base 10:1,
axis background/.style={fill=white},
legend style={legend pos= north west, legend cell align=left, align=left, draw=white!15!black}
]
\addplot [color=mycolor1]
  table[row sep=crcr]{%
-7	0.0983616\\
-6.92964824120603	0.0983616\\
-6.85929648241206	0.0983616\\
-6.78894472361809	0.0983616\\
-6.71859296482412	0.0983616\\
-6.64824120603015	0.0983616\\
-6.57788944723618	0.0983616\\
-6.50753768844221	0.0983616\\
-6.43718592964824	0.0983616\\
-6.36683417085427	0.0983616\\
-6.2964824120603	0.0983616\\
-6.22613065326633	0.0983616\\
-6.15577889447236	0.0983616\\
-6.08542713567839	0.0983616\\
-6.01507537688442	0.0983616\\
-5.94472361809045	0.0983616\\
-5.87437185929648	0.0983616\\
-5.80402010050251	0.0983616\\
-5.73366834170854	0.0983616\\
-5.66331658291457	0.0983616\\
-5.5929648241206	0.0983616\\
-5.52261306532663	0.0983616\\
-5.45226130653266	0.0983616\\
-5.38190954773869	0.0983616\\
-5.31155778894472	0.0983616\\
-5.24120603015075	0.0983616\\
-5.17085427135678	0.0983616\\
-5.10050251256281	0.0983616\\
-5.03015075376884	0.0983616\\
-4.95979899497487	0.0983616\\
-4.8894472361809	0.0983616\\
-4.81909547738693	0.0983616\\
-4.74874371859296	0.0983616\\
-4.678391959799	0.0983616\\
-4.60804020100502	0.0983616\\
-4.53768844221106	0.0983616\\
-4.46733668341709	0.0983616\\
-4.39698492462312	0.0983616\\
-4.32663316582915	0.0983616\\
-4.25628140703518	0.0983616\\
-4.18592964824121	0.0983616\\
-4.11557788944724	0.0983616\\
-4.04522613065327	0.0983616\\
-3.9748743718593	0.0983616010103748\\
-3.90452261306533	0.098361654710938\\
-3.83417085427136	0.098361882852626\\
-3.76381909547739	0.0983624062705279\\
-3.69346733668342	0.098363339037287\\
-3.62311557788945	0.0983647885871824\\
-3.55276381909548	0.0983668558402104\\
-3.48241206030151	0.0983696353261665\\
-3.41206030150754	0.0983732153087267\\
-3.34170854271357	0.0983776779095289\\
-3.2713567839196	0.0983830992322548\\
-3.20100502512563	0.0983895494867116\\
-3.13065326633166	0.0983970931129128\\
-3.06030150753769	0.0984057889051607\\
-2.98994974874372	0.0984156901361276\\
-2.91959798994975	0.098426844680937\\
-2.84924623115578	0.0984392951412458\\
-2.77889447236181	0.0984530789693256\\
-2.70854271356784	0.098468228592144\\
-2.63819095477387	0.0984847715354469\\
-2.5678391959799	0.098502730547839\\
-2.49748743718593	0.0985221237248666\\
-2.42713567839196	0.0985429646330982\\
-2.35678391959799	0.0985652624342065\\
-2.28643216080402	0.09858902200905\\
-2.21608040201005	0.0986142440817543\\
-2.14572864321608	0.098640925343794\\
-2.07537688442211	0.0986690585780742\\
-2.00502512562814	0.0986986327830118\\
-1.93467336683417	0.0987296332966174\\
-1.8643216080402	0.0987620419205768\\
-1.79396984924623	0.0987958370443324\\
-1.72361809045226	0.0988309937691649\\
-1.65326633165829	0.098867484032275\\
-1.58291457286432	0.0989052767308648\\
-1.51256281407035	0.0989443378462193\\
-1.44221105527638	0.0989846305677882\\
-1.37185929648241	0.0990261154172674\\
-1.30150753768844	0.0990687503726804\\
-1.23115577889447	0.0991124909924602\\
-1.1608040201005	0.0991572905395306\\
-1.09045226130653	0.0992031001053877\\
-1.02010050251256	0.0992498687341819\\
-0.949748743718593	0.0992975435467991\\
-0.879396984924623	0.0993460698649423\\
-0.809045226130653	0.0993953913352134\\
-0.738693467336684	0.0994454500531946\\
-0.668341708542713	0.0994961866875298\\
-0.597989949748744	0.0995475406040066\\
-0.527638190954774	0.0995994499896377\\
-0.457286432160804	0.0996518519767422\\
-0.386934673366834	0.0997046827670275\\
-0.316582914572864	0.0997578777556709\\
-0.246231155778895	0.0998113716554008\\
-0.175879396984925	0.0998650986205787\\
-0.105527638190955	0.0999189923712806\\
-0.0351758793969852	0.0999729863173783\\
0.0351758793969852	0.100027013682622\\
0.105527638190955	0.100081007628719\\
0.175879396984925	0.100134901379421\\
0.246231155778895	0.100188628344599\\
0.316582914572864	0.100242122244329\\
0.386934673366834	0.100295317232972\\
0.457286432160804	0.100348148023258\\
0.527638190954774	0.100400550010362\\
0.597989949748744	0.100452459395993\\
0.668341708542713	0.10050381331247\\
0.738693467336684	0.100554549946805\\
0.809045226130653	0.100604608664787\\
0.879396984924623	0.100653930135058\\
0.949748743718593	0.100702456453201\\
1.02010050251256	0.100750131265818\\
1.09045226130653	0.100796899894612\\
1.1608040201005	0.100842709460469\\
1.23115577889447	0.10088750900754\\
1.30150753768844	0.10093124962732\\
1.37185929648241	0.100973884582733\\
1.44221105527638	0.101015369432212\\
1.51256281407035	0.101055662153781\\
1.58291457286432	0.101094723269135\\
1.65326633165829	0.101132515967725\\
1.72361809045226	0.101169006230835\\
1.79396984924623	0.101204162955668\\
1.8643216080402	0.101237958079423\\
1.93467336683417	0.101270366703383\\
2.00502512562814	0.101301367216988\\
2.07537688442211	0.101330941421926\\
2.14572864321608	0.101359074656206\\
2.21608040201005	0.101385755918246\\
2.28643216080402	0.10141097799095\\
2.35678391959799	0.101434737565794\\
2.42713567839196	0.101457035366902\\
2.49748743718593	0.101477876275133\\
2.5678391959799	0.101497269452161\\
2.63819095477387	0.101515228464553\\
2.70854271356784	0.101531771407856\\
2.77889447236181	0.101546921030674\\
2.84924623115578	0.101560704858754\\
2.91959798994975	0.101573155319063\\
2.98994974874372	0.101584309863872\\
3.06030150753769	0.101594211094839\\
3.13065326633166	0.101602906887087\\
3.20100502512563	0.101610450513288\\
3.2713567839196	0.101616900767745\\
3.34170854271357	0.101622322090471\\
3.41206030150754	0.101626784691273\\
3.48241206030151	0.101630364673833\\
3.55276381909548	0.10163314415979\\
3.62311557788945	0.101635211412818\\
3.69346733668342	0.101636660962713\\
3.76381909547739	0.101637593729472\\
3.83417085427136	0.101638117147374\\
3.90452261306533	0.101638345289062\\
3.9748743718593	0.101638398989625\\
4.04522613065327	0.1016384\\
4.11557788944724	0.1016384\\
4.18592964824121	0.1016384\\
4.25628140703518	0.1016384\\
4.32663316582915	0.1016384\\
4.39698492462312	0.1016384\\
4.46733668341709	0.1016384\\
4.53768844221105	0.1016384\\
4.60804020100502	0.1016384\\
4.678391959799	0.1016384\\
4.74874371859296	0.1016384\\
4.81909547738693	0.1016384\\
4.8894472361809	0.1016384\\
4.95979899497488	0.1016384\\
5.03015075376884	0.1016384\\
5.10050251256281	0.1016384\\
5.17085427135678	0.1016384\\
5.24120603015075	0.1016384\\
5.31155778894472	0.1016384\\
5.38190954773869	0.1016384\\
5.45226130653266	0.1016384\\
5.52261306532663	0.1016384\\
5.5929648241206	0.1016384\\
5.66331658291457	0.1016384\\
5.73366834170854	0.1016384\\
5.80402010050251	0.1016384\\
5.87437185929648	0.1016384\\
5.94472361809045	0.1016384\\
6.01507537688442	0.1016384\\
6.08542713567839	0.1016384\\
6.15577889447236	0.1016384\\
6.22613065326633	0.1016384\\
6.2964824120603	0.1016384\\
6.36683417085427	0.1016384\\
6.43718592964824	0.1016384\\
6.50753768844221	0.1016384\\
6.57788944723618	0.1016384\\
6.64824120603015	0.1016384\\
6.71859296482412	0.1016384\\
6.78894472361809	0.1016384\\
6.85929648241206	0.1016384\\
6.92964824120603	0.1016384\\
7	0.1016384\\
};
\addlegendentry{$h$}

\addplot [color=mycolor2, mark=o, draw=none,mark options={solid, mycolor2},mark size=3]
  table[row sep=crcr]{%
3.15992681827936	0.101603902121274\\
1.86874695084948	0.101266365393683\\
1.48062481965115	0.100931063891409\\
0.810673231569611	0.100531621447911\\
0.344965668237285	0.10020115995411\\
-0.118937668970855	0.0998074626549745\\
-0.60384861698185	0.0994817371241436\\
-1.14023085964241	0.0990936619069045\\
-1.70423089229585	0.0987725712545315\\
-2.85269811315541	0.0983899985465015\\
};
\addlegendentry{$h^{\app}(X^{\star,\app})$, $I=10$}

\addplot [color=mycolor3, mark=triangle, draw=none,mark options={solid, mycolor3}, mark size=3]
  table[row sep=crcr]{%
3.15992681827936	0.101603902121274\\
3.46447123055149	0.101536214885594\\
1.55134808263681	0.10140111060852\\
2.2284018648142	0.101266365393683\\
2.36437054287652	0.101199126987588\\
1.44045905721514	0.101064917688161\\
1.32702463658059	0.100931063891409\\
1.05524150071383	0.100864269865416\\
1.11045228178805	0.100730946680037\\
0.869979568850857	0.100597975484666\\
0.763585894568397	0.100531621447911\\
0.49692021009269	0.100399175629639\\
0.391690881166576	0.100267078334831\\
0.296233756613623	0.10020115995411\\
0.126588550471512	0.100069582870748\\
0.000892483518736248	0.0999383508896794\\
-0.191771457691838	0.0998074626549745\\
-0.318576112317303	0.0997421470207209\\
-0.460970968497254	0.0996117718786964\\
-0.665107343576395	0.0994817371241436\\
-0.743005893616796	0.0994168469758655\\
-0.927710101490295	0.099287320308495\\
-1.12702950782785	0.099158130713324\\
-1.19607959831642	0.0990936619069045\\
-1.43546084331976	0.098964975458315\\
-1.66635885756229	0.0988366228095665\\
-1.77109408146675	0.0987725712545315\\
-2.0887532179939	0.0986447168756127\\
-2.43905998452328	0.0985171930663978\\
-3.07730682127565	0.0983899985465015\\
};
\addlegendentry{$h^{\app}(X^{\star,\app})$, $I=30$}

\end{axis}

\end{tikzpicture}%
}
\end{center}
\caption{\label{6_choixrho} Representation of $h$ and $h^{\app}$ with post processing measurements discretized with $I=10$ or $I=30$ points. Due to the bad condition number of $T$, the use of too many frequency points diminishes the quality of the reconstruction: the error $\Vert h-h^\app\Vert_\infty/h_{\max}$ is $2.7\%$ when $I=10$ and $7.9\%$ when $I=30$. Here, $x^{\text{meas}}=6$, $h$ defined in \eqref{6_h1} and sources in \eqref{6_eqmatlab} are defined in \eqref{6_sourceincr}.} 
\end{figure}

\begin{proof}
For every $1\leq i \leq I$, we see that 
\begin{equation}\nonumber
|x_i^\app-x_i|\leq \left|\sum_{j=1}^{i} V^\app_j-V_j\right|\leq \left|\sum_{j=1}^i T^{-1}(d_j-(TV)_j)\right|.
\end{equation}
We denote $a=d-TV$, and $b=T^{-1}a$. We see that 
\begin{equation}\nonumber
b_1=\frac{a_1}{p_{1,0}},\qquad b_i=\frac{4}{3p_{i,i-1}}\left[a_i-p_{i,0}a_1-\frac{1}{4}\sum_{j=2}^{i-1} (p_{i,j}+3p_{i,j-1})a_j\right],
\end{equation}
Using the definition of $p_{i,j}$ given in \eqref{pij}, we notice that $p_{i,j}=\cO(\rho)$ and it follows that 
\begin{equation}\nonumber
b_1=\cO(\Vert a\Vert_\infty), \qquad b_i=\cO(\rho^{-1/2}I\Vert a\Vert_\infty)=\cO(\rho^{-3/2}\Vert a\Vert_{\infty}),
\end{equation}
\begin{equation}\nonumber
\left|\sum_{j=1}^i b_j\right|=\cO(\Vert a\Vert_\infty \rho^{-3/2}I)=\cO(\Vert a\Vert_\infty \rho^{-5/2}).
\end{equation}
We conclude using Proposition \ref{6_mMV} and the fact that 
\begin{equation}\nonumber
|h^{\app}(x_i^\app)-h(x_i^\app)|\leq |h(x_i)-h(x^\app_i)|\leq \eta |x_i-x_i^\app|.
\end{equation}  

\end{proof}

To summarize, we have presented here a layer stripping method to recover an approximation of the width $h$. After processing the measurements in Section 3.1, we get rid of the periodicity of the data in Section 3.2 using a small discretization step $\rho$. Then, we solve a triangular system to provide a reconstruction of each $x_i$ in Section 3.3 and an approximation of $h$ in Section 3.4. The triangular system been not well-conditioned, we need to ensure that the discretization step $\rho$ is not too small if we want a controlled reconstruction error. The following section aims at applying this reconstruction method to numerical data.

\section{Numerical computations}

In this section, we show some numerical applications of our reconstruction method on slowly varying waveguides. We simulate section measurements using numerically generated data, and we provide reconstructions of increasing waveguides with different shape profiles. 

\subsection{Generation of data \label{6_num}}
In the following, numerical solutions of \eqref{6_eqmatlab} are generated using the software Matlab to solve numerically the equation in the waveguide $\Omega$. In every numerical simulation, we assume that $h'$ is supported between $x=-7$ and $x=7$. To generate the solution $u$ of \eqref{6_eqmatlab} on $\Omega_7$, we use a self-coded finite element method and a perfectly matched layer \cite{berenger1} on the left side of the waveguide between $x=-15$ and $x=-8$ and on the right side between $x=8$ and $x=15$. The coefficient of absorption for the perfectly matched layer is defined as $\alpha=-k((x-8)\textbf{1}_{x\geq 8}-(x+8)\textbf{1}_{x\leq -8})$ and the structured mesh is built with a stepsize of $10^{-3}$. 

\subsection{Method of reconstruction}

In the following, all the numerical measurements are generated following the process described in section \ref{6_num}. Using the previous results, we present all the steps to reconstruct $h^{\app}$, an approximation of $h$. 
\begin{enumerate}
\item Find an approximation of $h_{\min}$ and $h_{\max}$ using the method described in Appendix A. 
\item Choose a set of frequencies $\g k \subset K^\text{res}_N$ with a very small step of discretization $\rho_1$, and sources $f$, $b_\top$, $b_\bot$. For every frequency $k\in \g k$, measure the $N$-th modal component $u_{k,N}(x^{\text{meas}})$ of the wavefield $u$ solution of \eqref{6_eqmatlab}. 
\item Process the measurements by multiplying them by $1/q(k)$ and applying $\Phi^{-1}$. Then, get rid of the modulus $\pi$ by straightening up the sequence, and compute the approximate value of $\ell$ using \eqref{6_ellapp}. The available data $(d_i)$ are then the approximations of $\tint_{x_i}^{x^{\text{meas}}}k_{i,N}$. 
\item Reduce the number of frequencies $k_i$ used in $\g k$ and keep the associated data $d_i$, in order to have a bigger step size denoted $\rho_2\geq \rho_1$. Solve $TV=d$, where $T$ is defined in \eqref{6_M}, to find an approximation of the distance between every $x_i^\app$. 
\item Compute $h^\app$ using \eqref{6_happ}. 
\end{enumerate}

\begin{rem}
We propose here to choose two values of $\rho$: the first one is very small to ensure the existence of the left inverse of $\Phi$ in Proposition \ref{6_ti}. The second one is bigger due to the cutting of some frequencies, which improve the precision of the reconstruction in Theorem \ref{6_th2}. 
\end{rem}

\subsection{Numerical results}

We now apply this method to reconstruct different profiles of slowly increasing waveguides. We present in Figure \ref{6_recr} the reconstruction $h^\app$ obtained for different profiles $h$:
\begin{equation}\label{6_h1}
h_1(x)=0.1+\gamma_1\left(\frac{x^5}{5}-32\frac{x^3}{3}+256x\right)\textbf{1}_{\{-4\leq x\leq 4\}}-\gamma_2\textbf{1}_{\{x<-4\}}+\gamma_2\textbf{1}_{\{x>4\}},
\end{equation}
\begin{equation}\label{6_h2}
h_2(x)=0.1+\gamma_3\left(\frac{x^5}{5}-2x^4+16\frac{x^3}{3}\right)\left(\textbf{1}_{\{0\leq x\leq 4\}}-\textbf{1}_{\{-4\leq x<0\}}\right)+ \gamma_4\left(\textbf{1}_{\{x>4\}}-\textbf{1}_{\{x<-4\}}\right),
\end{equation}
\begin{equation}\label{6_h3}
h_3(x)=0.1+\gamma_5 x\textbf{1}_{\{-4\leq x\leq 4\}}+4\gamma_5\textbf{1}_{\{x>4\}}-4\gamma_4\textbf{1}_{\{x<-4\}},
\end{equation}
\begin{equation}\label{6_h4}
h_4(x)=0.1-4\gamma_5+4\gamma_5\frac{\sqrt{x+4}}{\sqrt{2}}\,\textbf{1}_{\{-4\leq x\leq 4\}}+8\gamma_5\textbf{1}_{\{x>4\}}.
\end{equation}
where $\gamma_1=3.10^{-6}$, $\gamma_2=8192/5.10^{-6}$, $\gamma_3=5.10^{-5}$, $\gamma_4=53/3.10^{-5}$, $\gamma_5=0.01/30$. All these profiles are represented in black in Figure \ref{6_recr}. The first two profiles are in $\mathcal{C}^2(\R)$ while $h_3$ and $h_4$ show corners, with an infinite derivative in $h_4$. Sources in \eqref{6_eqmatlab} are defined by 
\begin{equation}\label{6_sourceincr}
f(x,y)=\delta_6(x)y, \qquad b_{\top}(x)=\delta_6(x), \qquad b_{\bot}=0,
\end{equation}
the sets of frequencies are defined by 
\begin{equation}\label{6_Ki}
\g k_1=\{30.92:31.93:50\}, \quad \g k_2=\{30.9:31.95:50\}, \quad  \g k_3=\g k_4=\{31.01:31.83:50\},
\end{equation}
and data $(d_i)$ are reduced to use only $12$ frequencies for the inversion. The initial profiles are represented in black, while the reconstructions $h^{\app}$ are plotted in red and slightly shifted for comparison purposes. We see that the reconstruction are satisfactory and that the relative $L^\infty$ error of the reconstruction $E_\infty(h^\app):=\Vert h-h^\app\Vert_\infty/h_{\max}$ is of the same order as the one presented in \cite{niclas1} and even better than the one in \cite{bonnetier1}.

\begin{figure}[h]
\begin{center}
\scalebox{0.8}{\input{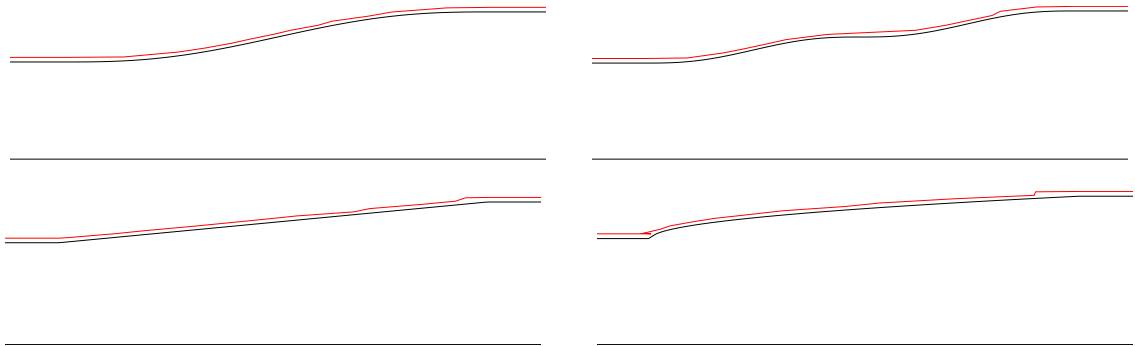}}\hspace{5mm} \scalebox{0.8}{
%
%
\begin{tikzpicture}

\begin{axis}[%
width=3.5in,
height=1in,
at={(0in,0in)},
scale only axis,
xmin=-5,
xmax=5,
ymin=0.092,
hide axis, 
ymax=0.102006666666667,
axis background/.style={fill=white},
legend style={legend cell align=left, align=left, draw=white!15!black}
]
\addplot [color=black]
  table[row sep=crcr]{%
-7	0.0982933333333333\\
-6.92964824120603	0.0982933333333333\\
-6.85929648241206	0.0982933333333333\\
-6.78894472361809	0.0982933333333333\\
-6.71859296482412	0.0982933333333333\\
-6.64824120603015	0.0982933333333333\\
-6.57788944723618	0.0982933333333333\\
-6.50753768844221	0.0982933333333333\\
-6.43718592964824	0.0982933333333333\\
-6.36683417085427	0.0982933333333333\\
-6.2964824120603	0.0982933333333333\\
-6.22613065326633	0.0982933333333333\\
-6.15577889447236	0.0982933333333333\\
-6.08542713567839	0.0982933333333333\\
-6.01507537688442	0.0982933333333333\\
-5.94472361809045	0.0982933333333333\\
-5.87437185929648	0.0982933333333333\\
-5.80402010050251	0.0982933333333333\\
-5.73366834170854	0.0982933333333333\\
-5.66331658291457	0.0982933333333333\\
-5.5929648241206	0.0982933333333333\\
-5.52261306532663	0.0982933333333333\\
-5.45226130653266	0.0982933333333333\\
-5.38190954773869	0.0982933333333333\\
-5.31155778894472	0.0982933333333333\\
-5.24120603015075	0.0982933333333333\\
-5.17085427135678	0.0982933333333333\\
-5.10050251256281	0.0982933333333333\\
-5.03015075376884	0.0982933333333333\\
-4.95979899497487	0.0982933333333333\\
-4.8894472361809	0.0982933333333333\\
-4.81909547738693	0.0982933333333333\\
-4.74874371859296	0.0982933333333333\\
-4.678391959799	0.0982933333333333\\
-4.60804020100502	0.0982933333333333\\
-4.53768844221106	0.0982933333333333\\
-4.46733668341709	0.0982933333333333\\
-4.39698492462312	0.0982933333333333\\
-4.32663316582915	0.0982933333333333\\
-4.25628140703518	0.0982933333333333\\
-4.18592964824121	0.0982933333333333\\
-4.11557788944724	0.0982933333333333\\
-4.04522613065327	0.0982933333333333\\
-3.9748743718593	0.0982933375233768\\
-3.90452261306533	0.098293557200071\\
-3.83417085427136	0.0982944750158659\\
-3.76381909547739	0.0982965427272223\\
-3.69346733668342	0.0983001581742027\\
-3.62311557788945	0.0983056673484978\\
-3.55276381909548	0.0983133664614529\\
-3.48241206030151	0.0983235040120937\\
-3.41206030150754	0.0983362828551527\\
-3.34170854271357	0.0983518622690955\\
-3.2713567839196	0.098370360024147\\
-3.20100502512563	0.0983918544503178\\
-3.13065326633166	0.0984163865054299\\
-3.06030150753769	0.0984439618431435\\
-2.98994974874372	0.0984745528809829\\
-2.91959798994975	0.098508100868363\\
-2.84924623115578	0.098544517954615\\
-2.77889447236181	0.0985836892570133\\
-2.70854271356784	0.0986254749288013\\
-2.63819095477387	0.0986697122272177\\
-2.5678391959799	0.0987162175815229\\
-2.49748743718593	0.0987647886610249\\
-2.42713567839196	0.0988152064431058\\
-2.35678391959799	0.098867237281248\\
-2.28643216080402	0.0989206349730604\\
-2.21608040201005	0.0989751428283045\\
-2.14572864321608	0.0990304957369206\\
-2.07537688442211	0.0990864222370546\\
-2.00502512562814	0.0991426465830833\\
-1.93467336683417	0.0991988908136414\\
-1.8643216080402	0.0992548768196474\\
-1.79396984924623	0.0993103284123298\\
-1.72361809045226	0.0993649733912534\\
-1.65326633165829	0.0994185456123455\\
-1.58291457286432	0.0994707870559224\\
-1.51256281407035	0.099521449894715\\
-1.44221105527638	0.0995702985618957\\
-1.37185929648241	0.0996171118191041\\
-1.30150753768844	0.0996616848244737\\
-1.23115577889447	0.0997038312006577\\
-1.1608040201005	0.0997433851028556\\
-1.09045226130653	0.0997802032868391\\
-1.02010050251256	0.0998141671769785\\
-0.949748743718593	0.0998451849342689\\
-0.879396984924623	0.0998731935243565\\
-0.809045226130653	0.0998981607855646\\
-0.738693467336684	0.0999200874969202\\
-0.668341708542713	0.0999390094461799\\
-0.597989949748744	0.0999549994978561\\
-0.527638190954774	0.0999681696612436\\
-0.457286432160804	0.0999786731584456\\
-0.386934673366834	0.0999867064923996\\
-0.316582914572864	0.0999925115149044\\
-0.246231155778895	0.0999963774946455\\
-0.175879396984925	0.0999986431852219\\
-0.105527638190955	0.0999996988931721\\
-0.0351758793969852	0.0999999885460004\\
0.0351758793969852	0.100000011454\\
0.105527638190955	0.100000301106828\\
0.175879396984925	0.100001356814778\\
0.246231155778895	0.100003622505355\\
0.316582914572864	0.100007488485096\\
0.386934673366834	0.1000132935076\\
0.457286432160804	0.100021326841554\\
0.527638190954774	0.100031830338756\\
0.597989949748744	0.100045000502144\\
0.668341708542713	0.10006099055382\\
0.738693467336684	0.10007991250308\\
0.809045226130653	0.100101839214435\\
0.879396984924623	0.100126806475644\\
0.949748743718593	0.100154815065731\\
1.02010050251256	0.100185832823021\\
1.09045226130653	0.100219796713161\\
1.1608040201005	0.100256614897144\\
1.23115577889447	0.100296168799342\\
1.30150753768844	0.100338315175526\\
1.37185929648241	0.100382888180896\\
1.44221105527638	0.100429701438104\\
1.51256281407035	0.100478550105285\\
1.58291457286432	0.100529212944078\\
1.65326633165829	0.100581454387654\\
1.72361809045226	0.100635026608747\\
1.79396984924623	0.10068967158767\\
1.8643216080402	0.100745123180353\\
1.93467336683417	0.100801109186359\\
2.00502512562814	0.100857353416917\\
2.07537688442211	0.100913577762945\\
2.14572864321608	0.100969504263079\\
2.21608040201005	0.101024857171696\\
2.28643216080402	0.10107936502694\\
2.35678391959799	0.101132762718752\\
2.42713567839196	0.101184793556894\\
2.49748743718593	0.101235211338975\\
2.5678391959799	0.101283782418477\\
2.63819095477387	0.101330287772782\\
2.70854271356784	0.101374525071199\\
2.77889447236181	0.101416310742987\\
2.84924623115578	0.101455482045385\\
2.91959798994975	0.101491899131637\\
2.98994974874372	0.101525447119017\\
3.06030150753769	0.101556038156856\\
3.13065326633166	0.10158361349457\\
3.20100502512563	0.101608145549682\\
3.2713567839196	0.101629639975853\\
3.34170854271357	0.101648137730905\\
3.41206030150754	0.101663717144847\\
3.48241206030151	0.101676495987906\\
3.55276381909548	0.101686633538547\\
3.62311557788945	0.101694332651502\\
3.69346733668342	0.101699841825797\\
3.76381909547739	0.101703457272778\\
3.83417085427136	0.101705524984134\\
3.90452261306533	0.101706442799929\\
3.9748743718593	0.101706662476623\\
4.04522613065327	0.101706666666667\\
4.11557788944724	0.101706666666667\\
4.18592964824121	0.101706666666667\\
4.25628140703518	0.101706666666667\\
4.32663316582915	0.101706666666667\\
4.39698492462312	0.101706666666667\\
4.46733668341709	0.101706666666667\\
4.53768844221105	0.101706666666667\\
4.60804020100502	0.101706666666667\\
4.678391959799	0.101706666666667\\
4.74874371859296	0.101706666666667\\
4.81909547738693	0.101706666666667\\
4.8894472361809	0.101706666666667\\
4.95979899497488	0.101706666666667\\
5.03015075376884	0.101706666666667\\
5.10050251256281	0.101706666666667\\
5.17085427135678	0.101706666666667\\
5.24120603015075	0.101706666666667\\
5.31155778894472	0.101706666666667\\
5.38190954773869	0.101706666666667\\
5.45226130653266	0.101706666666667\\
5.52261306532663	0.101706666666667\\
5.5929648241206	0.101706666666667\\
5.66331658291457	0.101706666666667\\
5.73366834170854	0.101706666666667\\
5.80402010050251	0.101706666666667\\
5.87437185929648	0.101706666666667\\
5.94472361809045	0.101706666666667\\
6.01507537688442	0.101706666666667\\
6.08542713567839	0.101706666666667\\
6.15577889447236	0.101706666666667\\
6.22613065326633	0.101706666666667\\
6.2964824120603	0.101706666666667\\
6.36683417085427	0.101706666666667\\
6.43718592964824	0.101706666666667\\
6.50753768844221	0.101706666666667\\
6.57788944723618	0.101706666666667\\
6.64824120603015	0.101706666666667\\
6.71859296482412	0.101706666666667\\
6.78894472361809	0.101706666666667\\
6.85929648241206	0.101706666666667\\
6.92964824120603	0.101706666666667\\
7	0.101706666666667\\
};

\addplot [color=red]
  table[row sep=crcr]{%
7	0.102006666666667\\
4	0.102006666666667\\
3.30032001224299	0.101969665164718\\
2.62571952909924	0.101688421277679\\
2.4646892756011	0.101408729081051\\
1.98995235151381	0.101061276403796\\
1.6242940906856	0.100785028901661\\
1.03775307586069	0.100441842327543\\
-0.617986569571076	0.100168976272155\\
-1.38506262090257	0.0998299777104709\\
-1.72987719300034	0.0995604313930424\\
-2.17096005792539	0.0992255446474519\\
-2.5550653137523	0.098959257851631\\
-3.22827538066648	0.0986284085630608\\
-4	0.0985933333333333\\
-7	0.0985933333333333\\
};

\addplot [color=black]
  table[row sep=crcr]{%
-7	0.092\\
7	0.092\\
};
\end{axis}

\end{tikzpicture}%
}\\[4mm]
\scalebox{0.8}{
%
%
\begin{tikzpicture}

\begin{axis}[%
width=3.5in,
height=1in,
at={(0in,0in)},
scale only axis,
xmin=-5,
xmax=5,
ymin=0.092,
ymax=0.102,
hide axis,
axis background/.style={fill=white},
legend style={legend cell align=left, align=left, draw=white!15!black}
]
\addplot [color=black]
  table[row sep=crcr]{%
-7	0.0986666666666667\\
-6.92964824120603	0.0986666666666667\\
-6.85929648241206	0.0986666666666667\\
-6.78894472361809	0.0986666666666667\\
-6.71859296482412	0.0986666666666667\\
-6.64824120603015	0.0986666666666667\\
-6.57788944723618	0.0986666666666667\\
-6.50753768844221	0.0986666666666667\\
-6.43718592964824	0.0986666666666667\\
-6.36683417085427	0.0986666666666667\\
-6.2964824120603	0.0986666666666667\\
-6.22613065326633	0.0986666666666667\\
-6.15577889447236	0.0986666666666667\\
-6.08542713567839	0.0986666666666667\\
-6.01507537688442	0.0986666666666667\\
-5.94472361809045	0.0986666666666667\\
-5.87437185929648	0.0986666666666667\\
-5.80402010050251	0.0986666666666667\\
-5.73366834170854	0.0986666666666667\\
-5.66331658291457	0.0986666666666667\\
-5.5929648241206	0.0986666666666667\\
-5.52261306532663	0.0986666666666667\\
-5.45226130653266	0.0986666666666667\\
-5.38190954773869	0.0986666666666667\\
-5.31155778894472	0.0986666666666667\\
-5.24120603015075	0.0986666666666667\\
-5.17085427135678	0.0986666666666667\\
-5.10050251256281	0.0986666666666667\\
-5.03015075376884	0.0986666666666667\\
-4.95979899497487	0.0986666666666667\\
-4.8894472361809	0.0986666666666667\\
-4.81909547738693	0.0986666666666667\\
-4.74874371859296	0.0986666666666667\\
-4.678391959799	0.0986666666666667\\
-4.60804020100502	0.0986666666666667\\
-4.53768844221106	0.0986666666666667\\
-4.46733668341709	0.0986666666666667\\
-4.39698492462312	0.0986666666666667\\
-4.32663316582915	0.0986666666666667\\
-4.25628140703518	0.0986666666666667\\
-4.18592964824121	0.0986666666666667\\
-4.11557788944724	0.0986666666666667\\
-4.04522613065327	0.0986666666666667\\
-3.9748743718593	0.0986750418760469\\
-3.90452261306533	0.0986984924623116\\
-3.83417085427136	0.0987219430485762\\
-3.76381909547739	0.0987453936348409\\
-3.69346733668342	0.0987688442211055\\
-3.62311557788945	0.0987922948073702\\
-3.55276381909548	0.0988157453936348\\
-3.48241206030151	0.0988391959798995\\
-3.41206030150754	0.0988626465661642\\
-3.34170854271357	0.0988860971524288\\
-3.2713567839196	0.0989095477386935\\
-3.20100502512563	0.0989329983249581\\
-3.13065326633166	0.0989564489112228\\
-3.06030150753769	0.0989798994974874\\
-2.98994974874372	0.0990033500837521\\
-2.91959798994975	0.0990268006700168\\
-2.84924623115578	0.0990502512562814\\
-2.77889447236181	0.0990737018425461\\
-2.70854271356784	0.0990971524288107\\
-2.63819095477387	0.0991206030150754\\
-2.5678391959799	0.09914405360134\\
-2.49748743718593	0.0991675041876047\\
-2.42713567839196	0.0991909547738694\\
-2.35678391959799	0.099214405360134\\
-2.28643216080402	0.0992378559463987\\
-2.21608040201005	0.0992613065326633\\
-2.14572864321608	0.099284757118928\\
-2.07537688442211	0.0993082077051926\\
-2.00502512562814	0.0993316582914573\\
-1.93467336683417	0.0993551088777219\\
-1.8643216080402	0.0993785594639866\\
-1.79396984924623	0.0994020100502513\\
-1.72361809045226	0.0994254606365159\\
-1.65326633165829	0.0994489112227806\\
-1.58291457286432	0.0994723618090452\\
-1.51256281407035	0.0994958123953099\\
-1.44221105527638	0.0995192629815745\\
-1.37185929648241	0.0995427135678392\\
-1.30150753768844	0.0995661641541039\\
-1.23115577889447	0.0995896147403685\\
-1.1608040201005	0.0996130653266332\\
-1.09045226130653	0.0996365159128978\\
-1.02010050251256	0.0996599664991625\\
-0.949748743718593	0.0996834170854271\\
-0.879396984924623	0.0997068676716918\\
-0.809045226130653	0.0997303182579565\\
-0.738693467336684	0.0997537688442211\\
-0.668341708542713	0.0997772194304858\\
-0.597989949748744	0.0998006700167504\\
-0.527638190954774	0.0998241206030151\\
-0.457286432160804	0.0998475711892797\\
-0.386934673366834	0.0998710217755444\\
-0.316582914572864	0.0998944723618091\\
-0.246231155778895	0.0999179229480737\\
-0.175879396984925	0.0999413735343384\\
-0.105527638190955	0.099964824120603\\
-0.0351758793969852	0.0999882747068677\\
0.0351758793969852	0.100011725293132\\
0.105527638190955	0.100035175879397\\
0.175879396984925	0.100058626465662\\
0.246231155778895	0.100082077051926\\
0.316582914572864	0.100105527638191\\
0.386934673366834	0.100128978224456\\
0.457286432160804	0.10015242881072\\
0.527638190954774	0.100175879396985\\
0.597989949748744	0.10019932998325\\
0.668341708542713	0.100222780569514\\
0.738693467336684	0.100246231155779\\
0.809045226130653	0.100269681742044\\
0.879396984924623	0.100293132328308\\
0.949748743718593	0.100316582914573\\
1.02010050251256	0.100340033500838\\
1.09045226130653	0.100363484087102\\
1.1608040201005	0.100386934673367\\
1.23115577889447	0.100410385259632\\
1.30150753768844	0.100433835845896\\
1.37185929648241	0.100457286432161\\
1.44221105527638	0.100480737018425\\
1.51256281407035	0.10050418760469\\
1.58291457286432	0.100527638190955\\
1.65326633165829	0.100551088777219\\
1.72361809045226	0.100574539363484\\
1.79396984924623	0.100597989949749\\
1.8643216080402	0.100621440536013\\
1.93467336683417	0.100644891122278\\
2.00502512562814	0.100668341708543\\
2.07537688442211	0.100691792294807\\
2.14572864321608	0.100715242881072\\
2.21608040201005	0.100738693467337\\
2.28643216080402	0.100762144053601\\
2.35678391959799	0.100785594639866\\
2.42713567839196	0.100809045226131\\
2.49748743718593	0.100832495812395\\
2.5678391959799	0.10085594639866\\
2.63819095477387	0.100879396984925\\
2.70854271356784	0.100902847571189\\
2.77889447236181	0.100926298157454\\
2.84924623115578	0.100949748743719\\
2.91959798994975	0.100973199329983\\
2.98994974874372	0.100996649916248\\
3.06030150753769	0.101020100502513\\
3.13065326633166	0.101043551088777\\
3.20100502512563	0.101067001675042\\
3.2713567839196	0.101090452261307\\
3.34170854271357	0.101113902847571\\
3.41206030150754	0.101137353433836\\
3.48241206030151	0.101160804020101\\
3.55276381909548	0.101184254606365\\
3.62311557788945	0.10120770519263\\
3.69346733668342	0.101231155778894\\
3.76381909547739	0.101254606365159\\
3.83417085427136	0.101278056951424\\
3.90452261306533	0.101301507537688\\
3.9748743718593	0.101324958123953\\
4.04522613065327	0.101333333333333\\
4.11557788944724	0.101333333333333\\
4.18592964824121	0.101333333333333\\
4.25628140703518	0.101333333333333\\
4.32663316582915	0.101333333333333\\
4.39698492462312	0.101333333333333\\
4.46733668341709	0.101333333333333\\
4.53768844221105	0.101333333333333\\
4.60804020100502	0.101333333333333\\
4.678391959799	0.101333333333333\\
4.74874371859296	0.101333333333333\\
4.81909547738693	0.101333333333333\\
4.8894472361809	0.101333333333333\\
4.95979899497488	0.101333333333333\\
5.03015075376884	0.101333333333333\\
5.10050251256281	0.101333333333333\\
5.17085427135678	0.101333333333333\\
5.24120603015075	0.101333333333333\\
5.31155778894472	0.101333333333333\\
5.38190954773869	0.101333333333333\\
5.45226130653266	0.101333333333333\\
5.52261306532663	0.101333333333333\\
5.5929648241206	0.101333333333333\\
5.66331658291457	0.101333333333333\\
5.73366834170854	0.101333333333333\\
5.80402010050251	0.101333333333333\\
5.87437185929648	0.101333333333333\\
5.94472361809045	0.101333333333333\\
6.01507537688442	0.101333333333333\\
6.08542713567839	0.101333333333333\\
6.15577889447236	0.101333333333333\\
6.22613065326633	0.101333333333333\\
6.2964824120603	0.101333333333333\\
6.36683417085427	0.101333333333333\\
6.43718592964824	0.101333333333333\\
6.50753768844221	0.101333333333333\\
6.57788944723618	0.101333333333333\\
6.64824120603015	0.101333333333333\\
6.71859296482412	0.101333333333333\\
6.78894472361809	0.101333333333333\\
6.85929648241206	0.101333333333333\\
6.92964824120603	0.101333333333333\\
7	0.101333333333333\\
};

\addplot [color=red]
  table[row sep=crcr]{%
7	0.101633333333333\\
4	0.101633333333333\\
3.59625102810985	0.101609018174453\\
3.41548844755507	0.101390801648246\\
2.7652856058328	0.101173523164968\\
1.80761962509402	0.100903234993889\\
1.50776337803909	0.100688045118395\\
0.435694961237349	0.100420348888087\\
-0.149891626723235	0.100207217649094\\
-0.914718615089875	0.0999420762542154\\
-1.53228016373605	0.0997309742511577\\
-2.36459083617067	0.0994683512912535\\
-2.95063428038811	0.09925924968076\\
-3.84473316177099	0.0989991094436002\\
-4	0.0989666666666667\\
-7	0.0989666666666667\\
};

\addplot [color=black]
  table[row sep=crcr]{%
-7	0.092\\
7	0.092\\
};

\end{axis}

\end{tikzpicture}%
} \hspace{5mm} \scalebox{0.8}{
%
%
\begin{tikzpicture}

\begin{axis}[%
width=3.5in,
height=1in,
at={(0in,0in)},
scale only axis,
xmin=-5,
xmax=5,
ymin=0.092,
ymax=0.101609018174453,
axis background/.style={fill=white},
axis x line*=bottom,
axis y line*=left,
hide axis,
legend style={legend cell align=left, align=left, draw=white!15!black}
]
\addplot [color=black]
  table[row sep=crcr]{%
-7	0.0986666666666667\\
-6.92964824120603	0.0986666666666667\\
-6.85929648241206	0.0986666666666667\\
-6.78894472361809	0.0986666666666667\\
-6.71859296482412	0.0986666666666667\\
-6.64824120603015	0.0986666666666667\\
-6.57788944723618	0.0986666666666667\\
-6.50753768844221	0.0986666666666667\\
-6.43718592964824	0.0986666666666667\\
-6.36683417085427	0.0986666666666667\\
-6.2964824120603	0.0986666666666667\\
-6.22613065326633	0.0986666666666667\\
-6.15577889447236	0.0986666666666667\\
-6.08542713567839	0.0986666666666667\\
-6.01507537688442	0.0986666666666667\\
-5.94472361809045	0.0986666666666667\\
-5.87437185929648	0.0986666666666667\\
-5.80402010050251	0.0986666666666667\\
-5.73366834170854	0.0986666666666667\\
-5.66331658291457	0.0986666666666667\\
-5.5929648241206	0.0986666666666667\\
-5.52261306532663	0.0986666666666667\\
-5.45226130653266	0.0986666666666667\\
-5.38190954773869	0.0986666666666667\\
-5.31155778894472	0.0986666666666667\\
-5.24120603015075	0.0986666666666667\\
-5.17085427135678	0.0986666666666667\\
-5.10050251256281	0.0986666666666667\\
-5.03015075376884	0.0986666666666667\\
-4.95979899497487	0.0986666666666667\\
-4.8894472361809	0.0986666666666667\\
-4.81909547738693	0.0986666666666667\\
-4.74874371859296	0.0986666666666667\\
-4.678391959799	0.0986666666666667\\
-4.60804020100502	0.0986666666666667\\
-4.53768844221106	0.0986666666666667\\
-4.46733668341709	0.0986666666666667\\
-4.39698492462312	0.0986666666666667\\
-4.32663316582915	0.0986666666666667\\
-4.25628140703518	0.0986666666666667\\
-4.18592964824121	0.0986666666666667\\
-4.11557788944724	0.0986666666666667\\
-4.04522613065327	0.0986666666666667\\
-3.9748743718593	0.0988161119465541\\
-3.90452261306533	0.0989579891484397\\
-3.83417085427136	0.0990505985422542\\
-3.76381909547739	0.0991248572185093\\
-3.69346733668342	0.0991886565885077\\
-3.62311557788945	0.0992454657468399\\
-3.55276381909548	0.0992971769932396\\
-3.48241206030151	0.0993449572892031\\
-3.41206030150754	0.099389586486098\\
-3.34170854271357	0.0994316163133967\\
-3.2713567839196	0.0994714541285087\\
-3.20100502512563	0.0995094108481618\\
-3.13065326633166	0.0995457301752652\\
-3.06030150753769	0.0995806073369814\\
-2.98994974874372	0.0996142015977173\\
-2.91959798994975	0.0996466449022754\\
-2.84924623115578	0.0996780480185674\\
-2.77889447236181	0.0997085050109785\\
-2.70854271356784	0.0997380965674927\\
-2.63819095477387	0.0997668925203394\\
-2.5678391959799	0.0997949537869493\\
-2.49748743718593	0.0998223338862854\\
-2.42713567839196	0.0998490801388527\\
-2.35678391959799	0.0998752346274835\\
-2.28643216080402	0.0999008349747263\\
-2.21608040201005	0.0999259149778922\\
-2.14572864321608	0.0999505051323741\\
-2.07537688442211	0.0999746330663635\\
-2.00502512562814	0.0999983239046381\\
-1.93467336683417	0.100021600575072\\
-1.8643216080402	0.100044484068521\\
-1.79396984924623	0.100066993660465\\
-1.72361809045226	0.100089147101072\\
-1.65326633165829	0.100110960779017\\
-1.58291457286432	0.100132449863349\\
-1.51256281407035	0.100153628426895\\
-1.44221105527638	0.100174509554068\\
-1.37185929648241	0.100195105435409\\
-1.30150753768844	0.10021542745081\\
-1.23115577889447	0.100235486243038\\
-1.1608040201005	0.100255291782895\\
-1.09045226130653	0.100274853427157\\
-1.02010050251256	0.100294179970244\\
-0.949748743718593	0.10031327969043\\
-0.879396984924623	0.100332160391282\\
-0.809045226130653	0.100350829438919\\
-0.738693467336684	0.10036929379559\\
-0.668341708542713	0.100387560050015\\
-0.597989949748744	0.100405634444849\\
-0.527638190954774	0.100423522901611\\
-0.457286432160804	0.100441231043343\\
-0.386934673366834	0.10045876421526\\
-0.316582914572864	0.100476127503596\\
-0.246231155778895	0.100493325752836\\
-0.175879396984925	0.100510363581507\\
-0.105527638190955	0.100527245396658\\
-0.0351758793969852	0.100543975407179\\
0.0351758793969852	0.100560557636049\\
0.105527638190955	0.10057699593164\\
0.175879396984925	0.100593293978141\\
0.246231155778895	0.100609455305203\\
0.316582914572864	0.100625483296862\\
0.386934673366834	0.100641381199818\\
0.457286432160804	0.100657152131106\\
0.527638190954774	0.100672799085239\\
0.597989949748744	0.100688324940845\\
0.668341708542713	0.100703732466846\\
0.738693467336684	0.100719024328227\\
0.809045226130653	0.100734203091418\\
0.879396984924623	0.100749271229318\\
0.949748743718593	0.100764231126006\\
1.02010050251256	0.100779085081136\\
1.09045226130653	0.10079383531407\\
1.1608040201005	0.100808483967745\\
1.23115577889447	0.100823033112312\\
1.30150753768844	0.100837484748545\\
1.37185929648241	0.100851840811058\\
1.44221105527638	0.100866103171325\\
1.51256281407035	0.100880273640528\\
1.58291457286432	0.100894353972245\\
1.65326633165829	0.100908345864978\\
1.72361809045226	0.100922250964547\\
1.79396984924623	0.100936070866353\\
1.8643216080402	0.100949807117506\\
1.93467336683417	0.100963461218855\\
2.00502512562814	0.100977034626899\\
2.07537688442211	0.100990528755599\\
2.14572864321608	0.101003944978097\\
2.21608040201005	0.10101728462835\\
2.28643216080402	0.101030549002675\\
2.35678391959799	0.101043739361222\\
2.42713567839196	0.101056856929366\\
2.49748743718593	0.101069902899044\\
2.5678391959799	0.101082878430009\\
2.63819095477387	0.101095784651043\\
2.70854271356784	0.101108622661096\\
2.77889447236181	0.101121393530379\\
2.84924623115578	0.101134098301411\\
2.91959798994975	0.101146737990005\\
2.98994974874372	0.101159313586218\\
3.06030150753769	0.10117182605526\\
3.13065326633166	0.101184276338352\\
3.20100502512563	0.101196665353557\\
3.2713567839196	0.10120899399657\\
3.34170854271357	0.10122126314147\\
3.41206030150754	0.101233473641448\\
3.48241206030151	0.101245626329494\\
3.55276381909548	0.101257722019066\\
3.62311557788945	0.10126976150472\\
3.69346733668342	0.101281745562725\\
3.76381909547739	0.101293674951639\\
3.83417085427136	0.101305550412878\\
3.90452261306533	0.101317372671248\\
3.9748743718593	0.101329142435464\\
4.04522613065327	0.101333333333333\\
4.11557788944724	0.101333333333333\\
4.18592964824121	0.101333333333333\\
4.25628140703518	0.101333333333333\\
4.32663316582915	0.101333333333333\\
4.39698492462312	0.101333333333333\\
4.46733668341709	0.101333333333333\\
4.53768844221105	0.101333333333333\\
4.60804020100502	0.101333333333333\\
4.678391959799	0.101333333333333\\
4.74874371859296	0.101333333333333\\
4.81909547738693	0.101333333333333\\
4.8894472361809	0.101333333333333\\
4.95979899497488	0.101333333333333\\
5.03015075376884	0.101333333333333\\
5.10050251256281	0.101333333333333\\
5.17085427135678	0.101333333333333\\
5.24120603015075	0.101333333333333\\
5.31155778894472	0.101333333333333\\
5.38190954773869	0.101333333333333\\
5.45226130653266	0.101333333333333\\
5.52261306532663	0.101333333333333\\
5.5929648241206	0.101333333333333\\
5.66331658291457	0.101333333333333\\
5.73366834170854	0.101333333333333\\
5.80402010050251	0.101333333333333\\
5.87437185929648	0.101333333333333\\
5.94472361809045	0.101333333333333\\
6.01507537688442	0.101333333333333\\
6.08542713567839	0.101333333333333\\
6.15577889447236	0.101333333333333\\
6.22613065326633	0.101333333333333\\
6.2964824120603	0.101333333333333\\
6.36683417085427	0.101333333333333\\
6.43718592964824	0.101333333333333\\
6.50753768844221	0.101333333333333\\
6.57788944723618	0.101333333333333\\
6.64824120603015	0.101333333333333\\
6.71859296482412	0.101333333333333\\
6.78894472361809	0.101333333333333\\
6.85929648241206	0.101333333333333\\
6.92964824120603	0.101333333333333\\
7	0.101333333333333\\
};

\addplot [color=red]
  table[row sep=crcr]{%
7	0.101633333333333\\
4	0.101633333333333\\
3.18774690884605	0.101609018174453\\
3.16621107781644	0.101390801648246\\
1.6776091242509	0.101173523164968\\
0.264697828060604	0.100903234993889\\
-0.360368972983363	0.100688045118395\\
-1.52710768338474	0.100420348888087\\
-2.08886920224065	0.100207217649094\\
-2.80500945817267	0.0999420762542154\\
-3.20232441848386	0.0997309742511577\\
-3.64121454533332	0.0994683512912535\\
-3.82379587475058	0.09925924968076\\
-4.14534140934738	0.0989991094436002\\
-4	0.0989666666666667\\
-7	0.0989666666666667\\
};

\addplot[color=black] 
table[row sep=crcr]{
-7 0.092 \\
7 0.092 \\};

\end{axis}

\end{tikzpicture}%
}
\end{center}
\caption{\label{6_recr} Reconstruction of four different increasing profiles. In black, the initial shape of $\Omega_5$, and in red, the reconstruction, slightly shifted for comparison purposes. In each case, $\g k=\g k_i$ is defined in \eqref{6_Ki}, $I$ is equal to $50$ for the post processing of the measurements and then to $12$ for the inversion of the data, $h=h_i$ is defined in \eqref{6_h1}, \eqref{6_h2}, \eqref{6_h3}, \eqref{6_h4}, and the sources of \eqref{6_eqmatlab} are defined in \eqref{6_sourceincr}. Top left: $i=1$, $E_\infty(h^\app)=0.97 \%$. Top right: $i=2$, $E_\infty(h^\app)=1.0 \%$. Bottom left: $i=3$, $E_\infty(h^\app)=1.1 \%$. Bottom right: $i=4$, $E_\infty(h^\app)=1.5 \%$.}
\end{figure}

In Figure \ref{6_bruit}, we illustrate the stability of the recovery by adding some artificial noise on measured data. We notice that the recovery is still satisfactory as long as the amplitude of the noise is small, and that the reconstruction error grows until it reaches a saturation error close to $100\%$.  

\begin{figure}[h]
\begin{center}
%
%
\definecolor{mycolor1}{rgb}{0.00000,0.44700,0.74100}%
\begin{tikzpicture}

\begin{loglogaxis}[%
width=8cm,
height=3cm,
scale only axis,
xmin= 0.00012341,
xmax= 54.7,
ymin=0,
xlabel={standart deviation of the noise $\sigma$},
ylabel={$E_\infty(h^\app)$},
grid=major,
ymax=1.2,
axis background/.style={fill=white},
legend style={legend cell align=left, align=left, draw=white!15!black}
]
\addplot [color=mycolor1]
  table[row sep=crcr]{%
0.00012340980408668	0.0202788233483712\\
0.000193211688616399	0.0207303373181742\\
0.000302494253955545	0.0210365695693689\\
0.000473588188848093	0.021598458330725\\
0.000741454654703551	0.0216746388260744\\
0.00116082921391838	0.0234209294264923\\
0.00181740643927192	0.0274711430065962\\
0.0028453506561554	0.0347095307243516\\
0.00445471094497023	0.0382357057129508\\
0.00697434235752547	0.0518770664890229\\
0.0109191038253323	0.0723565322552485\\
0.0170950639123326	0.079605242927591\\
0.0267642120490453	0.12663523022934\\
0.0419023321749326	0.146561392625418\\
0.0656027324279484	0.205486743975215\\
0.102708328597224	0.253038463073329\\
0.160801240631574	0.340494546280243\\
0.25175211535233	0.453157259857477\\
0.394145762404821	0.547102827440753\\
0.617078755442682	0.651946835741835\\
0.966104996525596	0.770722046705654\\
1.51254415433918	0.881817983805031\\
2.36805505307725	0.962276205444709\\
3.70745192351404	1.02174432596336\\
5.80442576590703	1.0574747367799\\
9.08747009185537	1.06772634630059\\
14.2274388545756	1.07063453442849\\
22.2746280664084	1.08391994754228\\
34.8733922224703	1.08252881280481\\
54.5981500331442	1.08234623495354\\
};

\end{loglogaxis}
\end{tikzpicture}%
\end{center}
\caption{\label{6_bruit} Relative reconstruction error $E_\infty(h^\app)$ with increasing additive noise on measured data. Given a standard deviation $\sigma$, we apply the inversion method with data $u_N(x^{\text{meas}})+\mathcal{N}$ where $\mathcal{N}$ is a uniform random Gaussian noise of standard deviation $\sigma$. The width profile $h$ is defined in \eqref{6_h4}, sources of \eqref{6_eqmatlab} are defined in \eqref{6_sourceincr}, $\g k$ is defined in \eqref{6_Ki} and $I=50$ for the post processing an $12$ for the inversion. }
\end{figure}
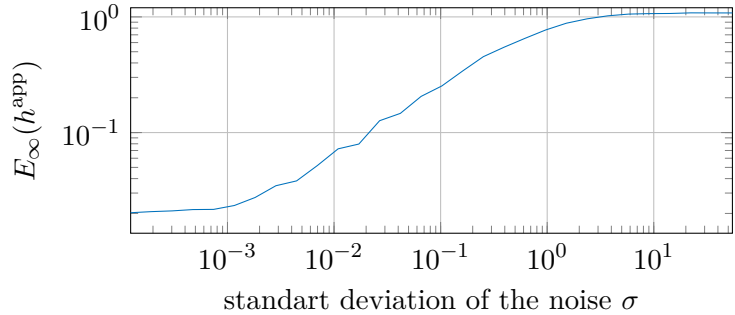

Finally, we illustrate in Figure \ref{6_recr5} that this method of reconstruction only works for increasing functions. We define 
\begin{equation}\label{6_h6}
h_6(x)=0.1-\gamma_7(x+5)\textbf{1}_{\{-5\leq x\leq 0\}} +\frac{\gamma_6}{4}(x-4)\textbf{1}_{\{0<x\leq 4\}},
\end{equation}
where $\gamma_6=25.10^{-4}$, $\gamma_7=5.10^{-4}$. The reconstruction of this function shows that if the width profile is not monotonous, we are only able to reconstruct the increasing part of the profile. To get the other part, one needs to move both the source and the receiver. 

\begin{figure}[h]
\begin{center}
\begin{tikzpicture}
\begin{axis}[width=5.5cm, height=2cm, axis on top, scale only axis, xmin=-7, xmax=7, ymin=0, ymax=0.1, colorbar left,axis y line=right,point meta min=0,point meta max=0.11, title={$|u|$},axis line style={draw=none},tick style={draw=none}, xlabel={$x$},ylabel={$y$}]
\addplot graphics [xmin=-7,xmax=7,ymin=0,ymax=0.1]{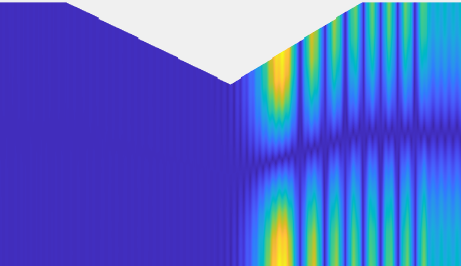};
\end{axis} 
\end{tikzpicture}
\hspace{2mm}
\raisebox{0.9cm}{\input{recr5}}
\end{center}
\caption{\label{6_recr5} Reconstruction of a non monotonous profile. On the left, we represent the wavefield $|u|$ at frequency $k=32.1$. On the right, we show the reconstruction of the width profile. In black is the initial shape of $\Omega_5$, and in red is the reconstruction slightly shifted for comparison purposes. We notice that we cannot reconstruct the left variations of the waveguide, which is explained by the fact that the wavefield $u$ never propagates in the left area of the waveguide. In the reconstruction, $\g k=\{31.42:32.1:50\}$, we only select $12$ frequencies for the inversion of data, $h$ is defined in \eqref{6_h6} and the sources of \eqref{6_eqmatlab} are defined in \eqref{6_sourceincr}.}
\end{figure}

To conclude, using a known source located far away at the right of the defect, we are able to reconstruct slowly increasing widths. Even if this method does not work for general shape defects, it is relevant to identify and localize obstruction types of defects and gives good numerically results. 

\section{Conclusion}

From the same analysis of the forward problem than in \cite{niclas1}, we have developed a new method to recover width defect using section measurements. We have used the study of the forward problem in slowly varying waveguide presented in \cite{bonnetier2} and the approximation of the solutions using the $\pi$-periodic function $\Phi$ to develop an inverse reconstruction method in slowly varying monotonous waveguides. Given wavefield measurements on a section of the waveguide for different locally resonant frequencies, we reconstruct with stability the associated resonant points which provides a good approximation of the width in the waveguide. 

Even if this method only works when sources and receivers are on the larger side of the waveguide, its reconstruction results are excellent and we believe that this method could be useful to detect and localize precisely the position of width defects in waveguides. Combined with a traditional multi-frequency back-scattering method like the one presented in \cite{bonnetier2}, this method could produce great results in recovering any type of width defect as long as there are small and slowly varying. 

\section*{Appendix A: Identification of $k_{\min}$ and $k_{\max}$}

Giving a compactly perturbed waveguide $\Omega$, we describe here how section measurements enable to approximate very precisely the quantities $k_{\min}$ and $k_{\max}$. The article \cite{bourgeois1} mentions that the problem \eqref{6_eqmatlab} is not well-defined when $k_n(x)=0$ in a non-trivial interval, which especially happens when $k=n\pi/h_{\min}$ or $k=n\pi/h_{\max}$. Numerically, this results in an explosion of the solution when $k$ tends to $n\pi/h_{\min}$ (resp. $n\pi/h_{\max}$) with a source term located in the area where $h(x)=h_{\min}$ (resp. $h(x)=h_{\max}$). Since our sources are located in the area $k=n\pi/h_{\max}$, measurements of the wavefield enable to find a good approximation of $h_{\max}$. 

Moreover, using Theorem \ref{6_th1}, we see that the Green function changes behavior when the mode $N$ switches from an evanescent mode to a locally resonant one. This change appends exactly around the frequency $N\pi/h_{\min}$ and comparing the measurements with the known behavior of evanescent modes, we can get a good approximation of $h_{\min}$. We illustrate it in Figure \ref{6_explosion}. 

\begin{figure}[h]
\begin{center}
%
%
\definecolor{mycolor1}{rgb}{0.00000,0.44700,0.74100}%
\definecolor{mycolor2}{rgb}{0.85000,0.32500,0.09800}%
\begin{tikzpicture}

\begin{axis}[%
width=12cm,
height=3cm,
at={(0in,0in)},
xlabel={$k$},
scale only axis,
grid=major,
xmin=29.5,
xmax=33.5,
ymin=0,
ymax=1,
axis background/.style={fill=white},
legend style={legend cell align=left, align=left, draw=white!15!black}
]
\addplot [color=mycolor1]
  table[row sep=crcr]{%
25	0.0252919402862436\\
25.3052631578947	0.0259317992865948\\
25.6105263157895	0.0266279914941542\\
25.9157894736842	0.0273888513479652\\
26.2210526315789	0.0282245484398488\\
26.5263157894737	0.0291476384516376\\
26.8315789473684	0.0301738424872994\\
27.1368421052632	0.0313231422094338\\
27.4421052631579	0.0326213871829885\\
27.7473684210526	0.0341027201793938\\
28.0526315789474	0.0358133295069878\\
28.3578947368421	0.0378175284570629\\
28.6631578947368	0.0402080525630839\\
28.9684210526316	0.043124524717331\\
29.2736842105263	0.0467891052744052\\
29.5789473684211	0.051582027286353\\
29.8842105263158	0.0582233960140804\\
30.1894736842105	0.0682997309894849\\
30.4947368421053	0.0863255305226047\\
30.8	0.135052548942379\\
30.9	0.185247723075446\\
30.9224489795918	0.206688882201917\\
30.9448979591837	0.237739788166959\\
30.9673469387755	0.288666105301945\\
30.9897959183673	0.396525104363943\\
31.0122448979592	0.959415890283245\\
31.034693877551	0.182415453842936\\
31.0571428571429	0.649034563872837\\
31.0795918367347	0.0720326418553468\\
31.1020408163265	0.40400119996645\\
31.1244897959184	0.039127330634189\\
31.1469387755102	0.282542766931676\\
31.169387755102	0.0448867070351987\\
31.1918367346939	0.398683519752413\\
31.2142857142857	0.101082404860801\\
31.2367346938776	0.212179572856249\\
31.2591836734694	0.145528603318402\\
31.2816326530612	0.163858830219906\\
31.3040816326531	0.169951621439464\\
31.3265306122449	0.105407511545464\\
31.3489795918367	0.224289495477002\\
31.3714285714286	0.017398821721832\\
31.3938775510204	0.220242081935317\\
31.4163265306122	0.0581168683478373\\
31.4387755102041	0.167422719592714\\
31.4612244897959	0.113452981180784\\
31.4836734693878	0.124598087385234\\
31.5061224489796	0.139163762932559\\
31.5285714285714	0.0863527235812852\\
31.5510204081633	0.15380179586134\\
31.5734693877551	0.0530480177527223\\
31.5959183673469	0.162430318811353\\
31.6183673469388	0.0340263696986374\\
31.6408163265306	0.162064893029248\\
31.6632653061224	0.0321699709739012\\
31.6857142857143	0.153611747399248\\
31.7081632653061	0.0461514178329456\\
31.730612244898	0.138221163320056\\
31.7530612244898	0.0719539134753954\\
31.7755102040816	0.113747327261964\\
31.7979591836735	0.101492007753016\\
31.8204081632653	0.0797780939352305\\
31.8428571428571	0.120046363161254\\
31.865306122449	0.0222250956956642\\
31.8877551020408	0.0723522212361531\\
31.9102040816327	0.0318254050156582\\
31.9326530612245	0.0447474800580919\\
31.9551020408163	0.0837604759029382\\
31.9775510204082	0.0703514666659632\\
32	0.0577190527657713\\
32.5	0.0504538073336781\\
32.7894736842105	0.0457854924348153\\
33.0789473684211	0.0431836791308336\\
33.3684210526316	0.0409196156060274\\
33.6578947368421	0.0378572655732873\\
33.9473684210526	0.0364336105426865\\
34.2368421052632	0.0342018309715657\\
34.5263157894737	0.032729456594581\\
34.8157894736842	0.0315247048171733\\
35.1052631578947	0.0302823043368353\\
35.3947368421053	0.0292254376289667\\
35.6842105263158	0.0282180160509859\\
35.9736842105263	0.0273196893828335\\
36.2631578947368	0.0264891839683646\\
36.5526315789474	0.0258090037339787\\
36.8421052631579	0.0251169137586261\\
37.1315789473684	0.0244197063101309\\
37.4210526315789	0.0237466360103117\\
37.7105263157895	0.0232978850219706\\
38	0.022715540539848\\
};
\addlegendentry{$|u_{k,N}(x_s)|$}

\addplot [color=mycolor2]
  table[row sep=crcr]{%
25	0.0272707979820263\\
25.3052631578947	0.0279158753218589\\
25.6105263157895	0.0286177398107332\\
25.9157894736842	0.029384805808766\\
26.2210526315789	0.0302273447153909\\
26.5263157894737	0.0311580466043531\\
26.8315789473684	0.0321928043413519\\
27.1368421052632	0.0333518324139197\\
27.4421052631579	0.0346613035269512\\
27.7473684210526	0.0361558121587041\\
28.0526315789474	0.0378822089227839\\
28.3578947368421	0.0399058088576579\\
28.6631578947368	0.0423209316380627\\
28.9684210526316	0.045269867292267\\
29.2736842105263	0.0489795888055377\\
29.5789473684211	0.0538398937030557\\
29.8842105263158	0.0605926501070041\\
30.1894736842105	0.0708850228068858\\
30.4947368421053	0.0894717920902721\\
30.8	0.141315863184209\\
30.9	0.198439736347782\\
30.9224489795918	0.22448810112551\\
30.9448979591837	0.264555452755014\\
30.9673469387755	0.338480270158445\\
30.9897959183673	0.562110602689991\\
31.0122448979592	0.645140339808629\\
31.034693877551	0.354124480834134\\
31.0571428571429	0.271664803158173\\
31.0795918367347	0.228638262379746\\
31.1020408163265	0.201157599838018\\
31.1244897959184	0.181665761157027\\
31.1469387755102	0.166916063444042\\
31.169387755102	0.155252078214551\\
31.1918367346939	0.145728571036544\\
31.2142857142857	0.137761628334949\\
31.2367346938776	0.130968497121844\\
31.2591836734694	0.125086439848525\\
31.2816326530612	0.119928315204243\\
31.3040816326531	0.115356729944984\\
31.3265306122449	0.111268232238097\\
31.3489795918367	0.107583239677247\\
31.3714285714286	0.104239394917708\\
31.3938775510204	0.101187050520374\\
31.4163265306122	0.0983861209282611\\
31.4387755102041	0.095803837830748\\
31.4612244897959	0.0934131176948337\\
31.4836734693878	0.0911913534570233\\
31.5061224489796	0.0891195059923205\\
31.5285714285714	0.0871814112500591\\
31.5510204081633	0.0853632450516171\\
31.5734693877551	0.083653104832047\\
31.5959183673469	0.0820406792785072\\
31.6183673469388	0.0805169848373616\\
31.6408163265306	0.0790741536607649\\
31.6632653061224	0.077705261530796\\
31.6857142857143	0.0764041871486047\\
31.7081632653061	0.075165496248364\\
31.730612244898	0.0739843455205591\\
31.7530612244898	0.0728564024631523\\
31.7755102040816	0.0717777781310667\\
31.7979591836735	0.0707449704004525\\
31.8204081632653	0.069754815858399\\
31.8428571428571	0.0688044488099523\\
31.865306122449	0.0678912661906008\\
31.8877551020408	0.06701289740441\\
31.9102040816327	0.0661671782908976\\
31.9326530612245	0.0653521285689107\\
31.9551020408163	0.0645659322216524\\
31.9775510204082	0.0638069203801006\\
32	0.0630735563372206\\
32.5	0.0512742766821067\\
32.7894736842105	0.0468311566967524\\
33.0789473684211	0.0433460335502597\\
33.3684210526316	0.0405163432389593\\
33.6578947368421	0.0381584668060754\\
33.9473684210526	0.0361536440692695\\
34.2368421052632	0.0344212003113405\\
34.5263157894737	0.0329041511064722\\
34.8157894736842	0.0315609444493395\\
35.1052631578947	0.0303604694498144\\
35.3947368421053	0.0292789065577654\\
35.6842105263158	0.0282976666266525\\
35.9736842105263	0.0274020001726132\\
36.2631578947368	0.0265800335748106\\
36.5526315789474	0.0258220854690514\\
36.8421052631579	0.0251201718746419\\
37.1315789473684	0.0244676414008471\\
37.4210526315789	0.0238589019559834\\
37.7105263157895	0.02328921300932\\
38	0.0227545255940862\\
};
\addlegendentry{$|G_{N}^\app(x_s,x_s)|$}

\addplot [color=black]
  table[row sep=crcr]{%
31.0025590814782	0\\
31.0025590814782	1\\
};
\addlegendentry{$N\pi/h_{\max}$}

\addplot [color=black, dash dot dot]
  table[row sep=crcr]{%
31.95	0\\
31.95	1\\
};
\addlegendentry{$N\pi/h_{\min}$}

\end{axis}

\end{tikzpicture}%
\caption{\label{6_explosion} Amplitude of $u_{k,N}(x^{\text{meas}})$ with respect to $k$ for $x^{\text{meas}}=6$ and a source $b=\delta_{x^{\text{meas}}}$, compared with the amplitude of the Green function $G_N^\app(x^{\text{meas}},x^{\text{meas}})$. For comparison purposes, $N\pi/h_{\max}$ and $N\pi/h_{\min}$ are plotted.}
\end{center}
\end{figure}
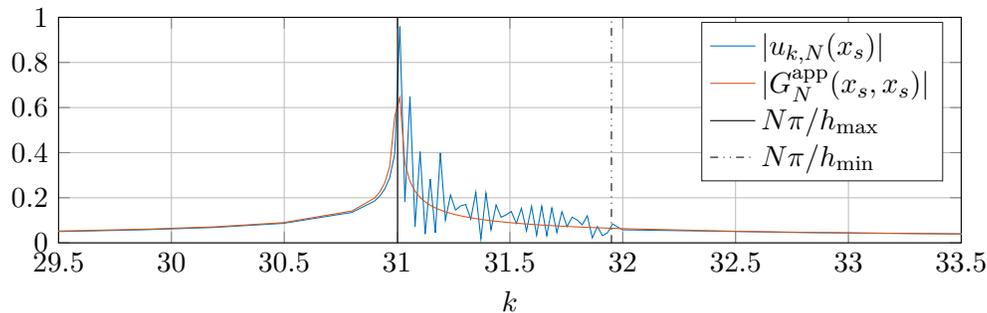

Measuring the section wavefield while $k$ varies and detecting its explosion and its changes of behavior provides a good approximation of the width at the left and the right of the waveguide.

\bibliographystyle{abbrv}
\bibliography{biblio}

\end{document}